\documentclass[11pt]{article}
\usepackage{xcolor}
\usepackage{amsmath,amssymb,amsfonts,amsthm,graphicx,graphics}
\usepackage{mathrsfs}
\usepackage{float}
\newtheorem{The}{Theorem}[section]

\newtheorem{Lem}[The]{Lemma}
\newtheorem{Cor}[The]{Corollary}
\theoremstyle{definition}
\newtheorem{Def}[The]{Definition}

\newcommand{\R}{\mathbb R}
\newcommand{\cond}{\mathrm{cond}}

\begin{document}

\date{\small\textsl{\today}}
\title{\Large 
Error and Stability Estimates of a Least-Squares Variational Kernel-Based Method for Second Order Elliptic PDEs
}
\author {\normalsize{SALAR SEYEDNAZARI $^{\mbox{\footnotesize a}}$\thanks{E-mail address:
\texttt{salarseyednazari@gmail.com},}\hspace{1mm}, MEHDI TATARI
 $^{\mbox{\footnotesize a,c}}$\thanks{E-mail
address: \texttt{mtatari@cc.iut.ac.ir}. (corresponding author). This research was in part supported by a grant from IPM (No.95650422)},\hspace{1mm}, DAVOUD MIRZAEI $^{\mbox{\footnotesize b}}$\thanks{E-mail
address: \texttt{d.mirzaei@sci.ui.ac.ir}.}}\\
\footnotesize{\em $^{\mbox{\footnotesize a}}$Department$\!$
of Mathematical Sciences,}\vspace{-2mm}\\
\footnotesize{\em Isfahan University of Technology, Isfahan, 84156-83111, Iran.}\\
\footnotesize{\em $^{\mbox{\footnotesize b}}$Department of Applied
Mathematics and Computer Science, Faculty of Mathematics and Statistics}\vspace{-2mm}\\
\footnotesize{\em Unievrsity of Isfahan, 81746-73441 Isfahan, Iran.}\\
\footnotesize{\em $^{\mbox{\footnotesize c}}$School of Mathematics, Institute for Research in Fundamental Sciences (IPM),}\\
\footnotesize{\em  P.O. Box: 19395-5746, Tehran, Iran. }\\
}\maketitle

\maketitle \thispagestyle{empty}



\begin{abstract}
We consider a least-squares variational kernel-based method for numerical solution of second order elliptic partial differential equations 
on a multi-dimensional domain.
In this setting it is not assumed that the differential operator is self-adjoint or positive definite as
it should be in the Rayleigh-Ritz setting. However, the new scheme leads to a symmetric and positive definite algebraic
system of equations. 
Moreover, the resulting method does not rely on certain subspaces satisfying the boundary conditions.
The trial space for discretization is provided via standard
kernels that reproduce the Sobolev spaces as their native spaces. 
The error analysis of the method is given, but it is partly subjected to an
inverse inequality on the boundary which is still an open problem.
The condition number of the final linear system is approximated in terms of the smoothness of the kernel and the discretization quality.
Finally, the results of some computational experiments support the theoretical error bounds.\\

$\mathbf{Keywords}$: Meshfree methods, Least-squares principles, Radial basis functions, Inverse inequalities, Error estimates.
\end{abstract}

\section{Introduction}\label{Int}
It is of interest to extend the theory of least-squares methods for numerical treatment of elliptic systems.
Some advantageous features are obtained via Least-Squares Principles (LSP) because of using the artificial energy functional to provide
a Rayleigh-Ritz-like setting; see \cite{AZIZ,PB3}.
One of the most attractive features of the least-squares methods is that the choice of approximating spaces
is not subject to the Ladyzhenskaya-Babuska-Brezzi (LBB) condition \cite{Brezzi}. Indeed,
the computation of stationary points, that is the paradigm of mixed-Galerkin methods, demands strict compatibility
{LBB condition} for continuous and discrete spaces,
if stable and accurate approximations are desired.
Furthermore, standard and mixed-Galerkin methods usually
produce nonsymmetric systems of algebraic equations
which must then be solved by direct
or non-robust iterative methods,
while least-squares methods involve only symmetric and positive definite systems.
The other motivation to extend the least-squares methods to PDE problems
of general boundary conditions, including nonhomogeneous
ones, is a greatly facilitated treatment with boundary conditions because their residuals can be incorporated into the least-squares functional \cite{PB3}.

The theory of least-squares methods
in numerical solution of elliptic boundary value problems was considered 
in Bramble and Schatz \cite{Bram2,Bram3} and Bramble and Nitsche \cite{Bram1}.
An extension to an elliptic equation of order
$2m$ {was} given in \cite{Bram3}, and an important simplifications in the analysis {was} presented in \cite{Ba}.
Also, a least-squares theory was developed for an elliptic system of Petrovsky type in \cite{W.W} and
 for elliptic systems of Agmon-Douglis-Nirenberg (ADN) type in \cite{AZIZ}.
We refer the reader to the survey articles \cite{PB1,PB2} and books \cite{PB3,Jia1} for more details.

In this paper we develop a least-squares method for numerical solution of
second order elliptic boundary value problems 
via reproducing kernels of Sobolev spaces $H^\tau(\Omega)$, $\Omega\subset \mathbb{R}^d$,  for some
$\tau  >d/2$, where $\Omega$ is a domain on which the PDE is posed. In particular we use radial basis function (RBF) approximations.
Although, we focus on the second order elliptic boundary value problems,
a generalization to higher order equations can be done in an obvious way.
The method involves the minimization of a least-squares functional that
consists of a weighted sum of the residuals occurring in the differential equation and the
boundary conditions.
One of the additional advantages is that the method provides a more accurate solution than one might be
expected from the approximating space. The method requires an approximating space
consisting of functions which are smooth enough to lie in the domain of the elliptic operator. Thus, for the success of the least-squares method
it is crucial to choose proper function spaces in which the boundary value problem is well-posed.

In a natural way,
straightforward least-squares methods for second or higher order differential equations require
finite dimensional subspaces of ${H^k}\left( \Omega  \right)$, $k \ge 2$. It is well known in the theory of finite element methods that the construction of such subspaces is much more difficult than those of
${H^1}\left( \Omega  \right)$.
Because the latter only need to be at most continuous
whereas, in practice, the former have to consist of $k$ times differentiable functions.
The ${C^k}\left( \Omega  \right)$ regularity requirement complicates finite element spaces in several ways.
First, it
cannot be satisfied unless the reference polynomial space is of a sufficiently high degree.
Second, unisolvency sets of ${C^k}\left( \Omega  \right)$ elements include both values of a function and its derivatives.
This fact greatly complicates the construction of bases and the assembly of the matrix problem.
Finally, ${C^k}\left( \Omega  \right)$ elements are not necessarily affine
equivalent because affine mappings do not necessarily preserve the normal direction.
To overcome this problem in finite element methods, 
the given problem is converted into a first order system.
Since in kernel-based method one can simply construct arbitrary smooth approximation spaces,
converting the problem into a first order system of equations may not be actually required. 
Thus, we directly apply RBF-based least-squares methods on the original equation. 
However, it is a good idea for a future study to apply the method on the corresponding first order system of equations because
smoothness of the basis functions affects stability of the final system.
On the other hand, the construction of the finite dimensional
subspaces using RBFs is independent of the problem dimension and
an extension to high dimensional problems is straightforward.

RBFs are powerful tools in multi-variable approximation and there exist substantial interest and effort for developing these basis functions.
This approximation is based on unrelated centers for the discretization process while most other methods are relied on underlying meshes.
In the present paper, we restrict ourselves to RBFs that reproduce Sobolev spaces as their native spaces.
We collect some few necessary results on RBFs while the whole theory is extensively presented in \cite{HW3}.

Both kernel-based collocation and Galerkin methods were investigated for solving PDEs.
The  unsymmetric  collocation method
was introduced by Kansa \cite{Kan}, in 1990. The linear system arising from this method may not be solvable in general; see Hon and Schaback \cite{Hon}. By changing the setting, a convergence analysis {was} given by Schaback \cite{SCH7}.
The symmetric collocation method was initially investigated by Wu \cite{ZW3} and Narcowich and Ward \cite{NAR5}.
The analysis of this method was investigated in \cite{Fra1,Fra2,HW3}, and recently in \cite{cheung-ling-schaback:2018}.
In \cite{HW2} an analysis for a meshless Galerkin method for a second order elliptic problem with natural boundary
conditions {was} given and a finite element like convergent estimate {was} obtained. 
A Petrov-Galerkin kernel-based method was given and analyzed in \cite{mirzaei:2018-1}.
Also, a meshless method for numerical solution of PDEs by using
 the Hermite-Birkhoff interpolation with radial basis is presented in see \cite{Z. Wu}.

Despite of numerous theoretical and computational advantages of LSPs,
there has not been a substantial effort devoted to investigating the least-squares
variational kernel-based approaches for solving PDEs; the subject that will be considered in this paper.

 The paper is organized as follows.
In Section \ref{sect-Notation}, some few notations and some auxiliary results are introduced. A short summary of the theory of RBFs approximation with a focus on those basis functions which generate Sobolev spaces is given.
The rest of this section is devoted to introduce the necessary technical framework of continuous and discrete least-squares principles
for numerical solution of second order differential equations.
Section \ref{sect-discretization} has three parts.
In the first part, the approximate solution is defined to be the minimizer of a mesh-dependent least-squares functional
that is a weighted sum of the least-squares residuals of differential equation and boundary conditions.
In the second part, the error analysis of the method is given. The analysis is partly based on an unproven inverse inequality on the boundary that demands a new 
research study. In the last part, the condition number of the final matrix is estimated.
Finally, in Section \ref{sect-numerical} some numerical results are reported to verify the theoretical bounds of the preceding section.

\section{Notations and Auxiliary Results}\label{sect-Notation}
In this paper, $\Omega$ will denote a simply connected bounded region in $\mathbb{R}^d$
with a sufficiently smooth boundary $\partial \Omega$,
and $C$ will be considered a generic positive constant whose meaning and value changes with context.
For $s \ge 0$, we use the standard notation and definition for the Sobolev spaces ${H^s}(\Omega )$ and
${H^s}(\partial \Omega )$ with corresponding inner products denoted by
${(\cdot,\cdot)_{s,\Omega }}$ and ${(\cdot,\cdot)_{s,\partial \Omega }}$
and norms by ${\left\| {\,\cdot\,} \right\|_{s,\Omega }}$ and ${\left\| {\,\cdot\,} \right\|_{s,\partial \Omega }}$,
respectively; see, e.g., \cite{Adam}, for details.
For $s<0$, the spaces ${H^s}(\Omega )$ and ${H^s}(\partial \Omega )$ are identified with the duals of
${H^{-s}}(\Omega )$ and ${H^{-s}}(\partial \Omega )$, respectively. A norm for $f \in H^s(\Omega)$, with $s<0$,  is defined by
$${\| f \|_{  s,\Omega }} = \mathop {\sup }\limits_{u \in {H^{-s}}(\Omega )} \frac{{{{\left( {f,u} \right)}_{0,\Omega }}}}{{{{\left\| u
 \right\|}_{-s,\Omega }}}}.$$
For $s<0$, the norm on $H^s(\partial \Omega)$  can be defined similarly; see \cite{Girault} for more details.

A function $f$ defined on $\Omega$
 is said to be \emph{Lipschitz continuous} if for some
constant $C$, there holds the inequality
\[\left| {f(x) - f(y)} \right| \le C\left\| {x - y} \right\|,\,\,\,\,\,\,\,\,\,\,\,\,\,\,\,\,\,\,\forall x,y \in \Omega .\]
In this formula, $\| {x - y} \|$ denotes the standard Euclidean distance between $x$ and $y$.
More generally, a function $f$ is said to be \emph{H\"{o}lder continuous} with exponent
$\beta \in (0, 1]$ if for some constant $C$,
\[\left| {f(x) - f(y)} \right| \le C\| {x - y} \|^{\beta},\,\,\,\,\,\,\,\,\,\,\,\,\,\,\,\,\,\,\forall x,y \in \Omega .\]
The H\"{o}lder space $ C^{0,\beta}(\overline \Omega)$ is defined to be the subspace of $C(\overline \Omega)$
functions
that are H\"{o}lder continuous with the exponent $\beta$.
For $l \in \mathbb{Z_+} $ and $\beta \in (0, 1]$, we similarly define the H\"{o}lder space
\[{C^{l,\beta }}(\overline \Omega  ) = \left\{ {f \in C^l(\overline \Omega  )\left| {{D^\alpha }f \in {C^{0,\beta }}(\overline \Omega  ),\,\,\,\,\left| \alpha  \right| = l} \right.} \right\},\]
where $\alpha=(\alpha_1,\ldots,\alpha_d)\in \mathbb N_0^d$ is a multi-index, $|\alpha|=\alpha_1+\cdots + \alpha_d$. The
partial derivative operator $D^\alpha$ is defined by
\begin{equation*}\label{eq2-4}
{D^\alpha } = \frac{\partial^{|\alpha|}}{\partial (x^1)^{\alpha_1}\cdots \partial (x^d)^{\alpha_d}}.
\end{equation*}
where $({x^1},\ldots,{x^d})^T\in \mathbb{R}^d$.
\\
\subsection{Approximation by RBFs}\label{sect-rbf}
For a given function space ${H^\tau }\left( \Omega  \right)$ on bounded domain
$\Omega  \subset {\mathbb{R}^d}$,
we define the finite dimensional kernel-based meshless trial spaces ${U_{\Phi ,X}} \subset {H^\tau }\left( \Omega  \right)$ by
$$
{U_{\Phi ,X}} := \mathrm{span}\left\{ {\Phi \left( {\,\cdot\, - {x_j}} \right)\,\,:\,\,{x_j} \in X} \right\},
$$
where $\Phi :{\mathbb{R}^d} \to \mathbb{R}$ is 
a radial basis function and
$$
X = \left\{ {{x_1},\ldots,{x_N}} \right\},
$$
will always be a finite subset of $\Omega$, with the
points all assumed to be distinct. There are two useful quantities associated with X. The first is the mesh norm for $X$ related to
$\Omega$, called \emph{ fill distance}, given by
$$
\,{h_{X,\Omega }}: = \mathop {\sup }\limits_{x \in \Omega } \,\mathop {\min }\limits_{{x_j} \in X} \,\left\| {x - {x_j}} \right\|,
$$
where norm $\left\| {\,\cdot\,} \right\|$ is the Euclidean norm in $\mathbb{R}^d$.
In other words, the largest ball in $\Omega$ that does not contain a data site has radius at most ${h_{X,\Omega }}$.
The second is the \emph{separation radius},
\begin{equation*}\label{eq1_3}
{q_X}: = \frac{1}{2}\mathop {\min }\limits_{{x_j} \ne {x_k}} \,\left\| {{x_j} - {x_k}} \right\|.
\end{equation*}
It is easy to see that if $\Omega$ is connected, we have ${h_{X,\Omega }} \ge {q_X}$.
A sequence of set points $\{X_k\}$
is called \emph{quasi-uniform} if there exists a uniform constant $\delta  > 0$ such that
 ${q_{X_k}} \ge \delta \,{h_{X_k,\Omega }}$ for all $k$.
In particular, the quantity $\rho_X := h_{X,\Omega}/q_X$ is commonly referred as the \emph{mesh ratio}
of $X$.
\begin{Def}\label{def2_2}
A continuous and even function $\Phi :{\mathbb{R}^d} \to \mathbb{R}$ is said to be
positive definite if for all $N \in \mathbb{N}$, all sets of pairwise
distinct centers $X = \left\{ {{x_1},\ldots,{x_N}} \right\}$ in $\mathbb{R}^d$, and all
$\alpha  \in {\mathbb{R}^N} \backslash  \left\{ 0 \right\}$
the quadratic form
$
\sum\nolimits_{j,k = 1}^N {{\alpha _j}{\alpha _k}} \Phi \left( {{x_j} - {x_k}} \right)
$
is strictly positive.
\end{Def}

The RBF interpolant of a continuous functions $u$ on a set $X$ is denoted by ${I_X}u$ and is given by
$$
{I_X}u := \sum\limits_{j = 1}^N {{b _j}\Phi \left( {\cdot - {x_j}} \right)},
$$
where the coefficient vector $b$ is determined by enforcing the interpolation conditions
$I_Xu(x_k) = u(x_k)$ for $k=1,\ldots,N$. If $\Phi$ is a positive definite kernel then the interpolation matrix
$B = (\Phi(x_k-x_j))$ is positive definite and the problem is {uniquely} solvable.

It is known that (see for example \cite{HW3}) a function $\Phi\in L^1(\mathbb R^d)\cap C(\mathbb R^d)$ is positive definite if and only if it is bounded and its Fourier transform is nonnegative and nonvanishing.
Our convention for the Fourier transform of a function $f\in L^1(\mathbb R^d)$ is
\begin{equation*}\label{eq1_8}
\widehat f(\omega): = (2\pi )^{ - d/2}\int_{{\mathbb{R}^d}} f(x){e^{ - i{\omega^T}x}} dx, \quad \omega\in \mathbb R^d.
\end{equation*}
In this paper we will further assume that
$\Phi$ has {an} algebraically decaying Fourier transform. To be more
precise, we assume that
\begin{equation}\label{eq1_12}
{C_{1}}{( {1 + {{\left\| \omega \right\|}^2}} )^{ - \tau }} \le \widehat \Phi (\omega) \le C_2{( {1 + {{\left\| \omega \right\|}^2}} )^{ - \tau }},\quad \omega \in {\mathbb{R}^d},
\end{equation}
where $C_1$ and $C_2$ are constants and $\tau  > d/2$.
By this assumption the \emph{native space}
\begin{equation*}\label{eq1_10}
{\mathcal{N}_\Phi }({\mathbb{R}^d}) := \left\{ {f \in {L^2}({\mathbb{R}^d}) \cap C({\mathbb{R}^d})\; :\;{\widehat f} /\sqrt {\widehat \Phi }  \in {L^2}({\mathbb{R}^d})} \right\},
\end{equation*}
 with the inner product
\begin{equation*}\label{eq1_11}
{(f,g)_{{\mathcal{N}_\Phi }({\mathbb{R}^d})}}: = (2\pi )^{ -d/2}\int_{\mathbb{R}^d} {\frac{{\widehat f(\omega)\overline {\widehat g(\omega)} }}{{\widehat \Phi (\omega)}}} d\omega,
\end{equation*}
is identical with the Sobolev space $H^\tau(\mathbb R^d)$ and their norms are equivalent \cite{HW3}.
Note that the inner product in $H^\tau(\mathbb R^d)$ is defined by
\begin{equation*}\label{eq1_13}
{\left( {f,g} \right)_{\tau ,{\mathbb{R}^d}}} := (2\pi )^{ - d/2}\int_{\mathbb{R}^d} \widehat f(\omega) \overline{\widehat g(\omega)} {(1 + {\left\| \omega \right\|^2})^\tau }d\omega,\quad f,g \in H^\tau (\mathbb{R}^d).
\end{equation*}
If we assume that $\Omega$ has a Lipschitz boundary to ensure the existence of a
continuous extension operator ${E_\Omega }:{H^\tau }(\Omega ) \to {H^\tau }({\mathbb{R}^d})$
then
the native space ${\mathcal{N}_\Phi }(\Omega )$ is norm-equivalent to ${H^\tau }(\Omega )$ \cite{HW3}.

It is well known that RBF interpolants are also the best approximants in the following sense
\begin{equation*}\label{eq1-15}
\mathop {\min }_{v \in {U_{\Phi ,X}}} \,{\| {u - v} \|_{{\mathcal{N}_\Phi }(\Omega )}} = {\left\| {u - {I_X}u} \right\|
_{{\mathcal{N}_\Phi }(\Omega )}}.
\end{equation*}
Hence, if the native space coincides with an appropriate Sobolev
space, the norm of $u - {I_X}u$ can be bounded by the norm of the target function $u$ in Sobolev spaces.
Since the smoothness of $u$ is unknown in general, we have to
look for convergence results where $\Phi$ can be chosen independent of the smoothness
of $u$; i.e., the error estimates include situations in which $u$ does not belong to the native space of the RBF.
In \cite{NAR1,NAR2,NAR3} the Sobolev type error estimates for positive real $\tau$,
for functions inside or outside the native space were derived.
\begin{The} \label{th2_1}
Suppose a positive definite kernel $\Phi$ satisfying (\ref{eq1_12}), with
$\tau  \ge k > d/2$,
 and let a bounded Lipschitz domain
 $\Omega  \subset {\mathbb{R}^d}$ be given. Furthermore, let
$X \subset \Omega$ has mesh norm ${h_{X,\Omega }}$. Then
there exists a function $v^h \in {U_{\Phi ,X}}$, a constant  $C$ independent of $u$ and
${h_{X,\Omega }}$ such that
\begin{equation*}\label{eq1_16}
{\| {u - v^h} \|_{r,\Omega }} \le C\,h_{X,\Omega }^{k - r}{\left\| u \right\|_{k,\Omega }},\quad 0 \le r \le k,
\end{equation*}
and
\begin{equation*}\label{eq1_17}
{\left\| {u - {I_X}u} \right\|_{r,\Omega }} \le C\,h_{X,\Omega }^{k - r}{\left\| u \right\|_{k,\Omega }}, \quad 0 \le r \le k,
\end{equation*}
for all
$u \in {H^k}(\Omega )$.
\end{The}

\subsection{CLSP for second order PDEs}\label{sect-adn2order}
For a bounded domain $\Omega  \subset {\mathbb{R}^d}$ with boundary $\partial \Omega$, we consider the
 following second order elliptic operator
\begin{align}
{{L}}u\left( x \right) =  - \sum_{i,j = 1}^d {{a_{ij}( x )}\frac{{{\partial ^2}u}}{{\partial {x^i}\partial {x^j}}}} \left( x \right) + \sum_{i = 1}^d {{b_{i}\left( x \right)}\frac{{\partial u}}{{\partial {x^i}}}} \left( x \right) + c\left( x \right)u\left( x \right) &= f\left( x \right),\quad x \in \Omega , \label{eq4-21} \\
{{B}}u\left( x \right) = u\left( x \right) &= g\left( x \right),\quad x \in \partial \Omega, \label{eq4-22}
\end{align}
where $u\in H^2(\Omega)$, $f \in {L^2}\left( \Omega  \right)$ and
$g \in H^{3/2}\left( {\partial \Omega } \right).$

If we assume that $U$, $V$ and $W$ are Hilbert spaces and
problem \eqref{eq4-21}-\eqref{eq4-22} is well-posed so that it
has a unique solution for all smooth data $f$ and $g$ and
there exist positive constants ${C_1}$ and ${C_2}$ such that
\begin{equation}\label{eq2-5}
{C_1}{\left\| {u} \right\|_U} \le {\left\| {{L}{u}} \right\|_V} + {\left\| {{B}{u}} \right\|_W} \le {C_2}{\left\| {u} \right\|_U}.
\end{equation}
This relation is called \emph{energy balance} which is fundamental
to least-squares methods because it defines a proper norm-equivalence
between solution space $U$ and data Space $V \times W$.

In order to achieve high order convergence, the regularity of $u$ needs
to be higher than what is strictly required by the problem itself. Here we assume for a real $k\geq 2$,
\begin{equation*}
U:= H^{k}\left( \Omega  \right) ,\quad V:= H^{k-2}\left( \Omega  \right) ,\quad W:= H^{k-1/2}\left(\partial \Omega  \right).
\end{equation*}

By using an Agmon-Douglis-Nirenberg (ADN) setting \cite{ADN1}, the left inequality in
\eqref{eq2-5} can be proved for the above $L$ and $B$ operators. See \cite{YAR,YAR2} for details and proofs.
Problem \eqref{eq4-21}-\eqref{eq4-22} is well-posed if and only if the boundary operator ${B}$ complements ${L}$
in a proper way. As specified in \cite{ADN1}, this is equivalent to an algebraic condition,
called the \emph{complementing condition}, on
the principal parts of $L$ and $B$. However, we shall not state these conditions here
as they are somewhat complicated and are not needed in the continuation.
But in what follows we assume ${L}$ is uniformly elliptic in the sense of ADN in $\overline \Omega$
 and ${B}$ satisfies the complementing condition.

\begin{Lem}\label{cor4_1}
Let $k \ge 2$ be real and assume $\Omega$ is a bounded domain such that $\,\partial \Omega  \in {C^{k}}.$ Furthermore,
assume that the ${a_{ij}}, {b_{i}},$ and $c$ are in ${C^{k-2}}\left( {\overline \Omega  } \right)$.
If $f \in {H^{k-2}}\left( \Omega  \right)$ and $g \in {H^{k-1/2}}\left( {\partial \Omega } \right)$
then every solution $u \in {H^2}\left( \Omega  \right)$ is indeed in ${H^{k}}\left( \Omega  \right)$.
Also, there exists a constant $C > 0$ independent of $u$, $f$ and $g$ such that for every solution $u\in {H^{k}}\left( \Omega  \right)$ we have
\begin{equation}\label{eq4-23}
\left\| u \right\|_{k,\Omega } \le C\left( \| f \|_{k-2,\Omega } + \| g \|_{k - 1/2,\partial \Omega }\right).
\end{equation}
Moreover, the a priori bound (\ref{eq4-23}) can be also extended to all real values $k<2$.
\end{Lem}

Throughout the paper and in what follows,
whenever we assume $u\in H^k(\Omega)$, for some real $k\geq 2$, is the unique solution of \eqref{eq4-21}-\eqref{eq4-22} (perhaps in a weak sense),
then $\Omega$ is assumed to be a bounded domain in $\mathbb{R}^d$ such that $\partial \Omega  \in {C^k}$. Furthermore,
the functions ${a_{ij}}$, $b_i$ and $c$ in \eqref{eq4-21} are assumed to be of class ${C^{k-2}}\left( {\overline \Omega  } \right)$.

Lemma \ref{cor4_1} yields the inverse of
 mapping $T:{H^{k}}\left( \Omega  \right) \to {H^{k-2}}\left( \Omega  \right) \times H^{k -1/2}\left( {\partial \Omega } \right)$
defined by $Tu = \left( {{{L}}u,u} \right)$ which is continuous for all real $k$.
To extend this a priori estimate to the energy balance, we need the trace theorem that relates the Sobolev norms
of functions on the interior of $\Omega$ with the Sobolev norms of their restrictions to the boundary $\partial \Omega$ \cite{Adam}.
\begin{The}(Trace Theorem)\label{lem4_2}
Assume that $\partial \Omega  \in {C^{\ell,1}}$ for some $\ell\ge 0$ and $1/2 < k \le \ell + 1$. Then, the trace operator
$\pi :{H^k}(\Omega ) \to {H^{k - 1/2}}(\partial \Omega )$,
where
$\pi u: = u\left| {_{\partial \Omega }} \right.$, is bounded. This means there exists a positive constant $C$
 such that for all
$u \in {H^k}\left( \Omega  \right)$
\begin{equation*}\label{eq4-24}
\left\| \pi u \right\|_{k - 1/2,\partial \Omega } \le C{\left\| u \right\|_{k,\Omega }}.
\end{equation*}
\end{The}
By applying Lemma \ref{cor4_1} and Theorem \ref{eq4-24}
we simply have the following theorem.
\begin{The}\label{th4_3}
For real $q \ge 0$, let $\Omega$ be a bounded domain such that $\partial \Omega  \in C^{q+2}$. Furthermore,
assume that the coefficients of ${L}$ are of class ${C^q}\left( {\overline \Omega  } \right)$. Then
the mapping $T:{H^{q+2}}\left( \Omega  \right) \to {H^{q}}\left( \Omega  \right) \times {H^{q +3/2}}\left( {\partial \Omega } \right)$
defined by $Tu = \left( {{{L}}u,u} \right)$ is a homeomorphism, and the norms
${\left\| {\,\cdot\,} \right\|_{q + 2,\Omega }}$
and
${\left\| {\,{{L}}\cdot\,} \right\|_{q,\Omega }} + {\left\| {\,\cdot\,} \right\|_{q + 3/2,\partial \Omega }}$
are equivalent; i.e., there exists a constant $C > 0$ independent of $u$ such that
\begin{equation}\label{eq4-25}
{C^{ - 1}}{\left\| u \right\|_{q + 2,\Omega }} \le {\left\| {{{L}}u} \right\|_{q,\Omega }} + {\left\| u \right\|_{q +3/2,\partial \Omega }} \le {C}{\left\| u \right\|_{q + 2,\Omega }}.
\end{equation}
\end{The}
This relation 
defines a proper norm-equivalence
between solution space ${H^{q + 2}}\left( \Omega  \right)$ and data Space 
${H^q}\left( \Omega  \right) \times {H^{q + 3/2}}\left( {\partial \Omega } \right)$. Moreover,
the energy balance (\ref{eq4-25}) allows to define a well-posed \emph{continuous least-squares principle} (CLSP)
 for (\ref{eq4-21})-(\ref{eq4-22}) by energy functional
\begin{equation}\label{eq4-26}
J_{q}\left( {u;f,g} \right): = \frac{1}{2}\left( {\left\| {{{L}}u - f\,} \right\|_{q,\Omega }^2 + \left\| {u - g} \right\|_{q + 3/2,\partial \Omega }^2} \right).
\end{equation}
The corresponding CLSP is given by the pair
 $\left\{ {{H^{q + 2}}\left( \Omega  \right),J_q} \right\}$,
which corresponds to an unconstrained minimization problems
\begin{equation}\label{eq4-27}
\mathop {\min }\limits_{u \in {H^{q + 2}}\left( \Omega  \right)} \,\,\,{J_q}\left( {u;f,g} \right).
\end{equation}
From Theorem \ref{th4_3} the least-squares functional $J_q(\cdot;0,0)$ 
 defines a norm-equivalent property 
 for $\|\cdot\|_{q+2,\Omega}$
,i.e. , obviously one sees that the functional $J_q\left( {\cdot\,;0,0} \right)$ is equivalent to
${\left\| \cdot \right\|_{q+2,\Omega}}$ in the sense that
\begin{equation}\label{eq2-8}
\frac{1}{4}C_1^2\left\| u \right\|_{q+2,\Omega}^2 \le J_q\left( {u;0,0} \right) \le \frac{1}{2}C_2^2\left\| u \right\|_{q+2,\Omega}^2.
\end{equation}
Therefore, according to \cite[Theorem 2.5]{PB1}, 
{for all real values $q\geq0$,}
problem (\ref{eq4-27})
has a unique minimizer $u \in {H^{q + 2}}\left( \Omega  \right)$ that depends continuously on the data
$(f,g) \in {H^q}\left( \Omega  \right) \times {H^{q + 3/2}}\left( {\partial \Omega } \right)$.
Moreover, it is not difficult to see that a minimizer of \eqref{eq4-21}-\eqref{eq4-22} solves (\ref{eq4-27}) and conversely; i.e.,
the problems \eqref{eq4-21}-\eqref{eq4-22} and 
(\ref{eq4-27}) are equivalent in the sense that $u \in {H^{q + 2}}\left( \Omega  \right)$ is a solution of (\ref{eq4-27}) if and only if it is
also a solution, perhaps in a generalized sense, of \eqref{eq4-21}-\eqref{eq4-22}.

The Euler-Lagrange equation for (\ref{eq4-27}) is then given by the variational problem:
\begin{equation*}\label{eq4-30}
\mathrm{seek} ~~u \in {H^{q+2}}\left( \Omega  \right)~
\mathrm{such~that} ~~Q_q\left( {u,v} \right) = F_q\left( v \right),\,\,\,\,\,\,\,\,\,\forall v \in {H^{q+2}}\left( {\Omega } \right),
\end{equation*}
where
\begin{equation*}\label{eq4-31}
Q_q\left( {u,v} \right) = {\left( {{{L}}u,{{L}}v} \right)_{q,\Omega }} + {\left( {u,v} \right)_{q+3/2,\partial \Omega }} ,
~~~~~~\mbox{and}~~~~~~
F_q\left( v \right) = {\left( {{{L}}v,f} \right)_{q,\Omega }} + {\left( {v,g} \right)_{q+3/2,\partial \Omega }} .
\end{equation*}
We notice that the \emph{energy inner product} associated with
 $\left\{ {{H^{q + 2}}\left( \Omega  \right),J_q} \right\}$ is given by
${((\cdot,\cdot))_q}:{H^{q + 2}}(\Omega ) \times {H^{q + 2}}(\Omega ) \to \mathbb{R}$, where
$$
{((u,v))_q} := Q_q (u,v),
$$
and \emph{energy norm} is defined by
$
{|||u||{|_q}} := ((u,u))_q^{1/2} = \left[2 J_q(u) \right]^{1/2}.
$
The norm-equivalence property
\begin{equation}\label{eq4-35}
C_1\left\| u \right\|_{q + 2,\Omega } \le |||u||{|_q} \le C_2{\left\| u \right\|_{q + 2,\Omega }},
\end{equation}
holds by (\ref{eq4-25}). The special case $q = 0$ in (\ref{eq4-26}) gives rise to the CLSP
\begin{equation*}\label{eq4-28}
J_0\left( u;f,g \right) = \frac{1}{2}\left( \| Lu - f \|_{0,\Omega }^2 + \| u - g \|_{3/2,\partial \Omega }^2 \right),
\end{equation*}
where its associated energy balance for all $u \in {H^2}(\Omega )$ is
\begin{equation}\label{eq4-321}
C_1\left\| u \right\|_{2,\Omega } \le {\left\| {{{L}}u} \right\|_{0,\Omega }} + {\left\| u \right\|_{3/2,\partial \Omega }} \le C{\left\| u \right\|_{2, \Omega }}.
\end{equation}

In what follows we may write $J_q\left( {u} \right)$ instead of $J_q\left( {u;0,0} \right)$,
$J_q$ instead of $J_q(\cdot;0,0)$ and $Q_q$ instead of $Q_q(\cdot,\cdot)$.

A least-squares discretization can be defined by choosing a family of finite subspaces ${U^h} \subset H^{q+2}(\Omega)$
parameterized by $h$ tending to zero and then restricting the unconstrained minimization problem (\ref{eq4-27}) to the subspaces.
Thus, the  approximation ${u^h} \in {U^h}$ to the solution $u \in H^{q+2}(\Omega)$ of (\ref{eq4-21})-(\ref{eq4-22})
is the solution of the following problem
\begin{equation}\label{eq2-12}
\mbox{seek ${u^h} \in {U^h}$ such that~~$J_q\left( {{u^h};f,g} \right) \le J_q\left( {{v^h};f,g} \right),\,\,\,\,\,\,\,\,\,\forall {u^h} \in U^h$}.
\end{equation}
This process leads to a discrete variational form given by
\begin{equation}\label{eq2-13}
\mbox{seek ${u^h} \in {U^h}$ such that~~$Q_q\left( {{u^h},{v^h}} \right) = F_q\left( {{v^h}} \right),\,\,\,\,\,\,\,\,\,\forall {v^h} \in U^h$}.
\end{equation}
If we choose a basis  $\left\{ {{\phi _j}} \right\}_{j = 1}^N$ and assume ${u^h} = \sum\nolimits_{j = 1}^N {{c_j}{\phi _j}} $
for some constants $\left\{ {{c_j}} \right\}_{j = 1}^N$, then the discretized problem (\ref{eq2-13}) is equivalent to the linear system
\begin{equation}\label{Ax=b}
Ac = b
\end{equation}
where $A$ is a symmetric matrix with entries ${a_{ij}} = Q_q\left( {{\phi _i},{\phi _j}} \right)$ and
 $b$ is a $N$ vector with ${b_i} = F_q\left( {{\phi _i}} \right)$ for all $i,j = 1,\ldots,N$.


Note that, in the above setting, one does not assumed that ${L}$ is positive definite and self-adjoint,
while in Rayleigh-Ritz setting does. However, not only LSP preserves  all attractive
features of a Rayleigh-Ritz setting but also it does not have some Rayleigh-Ritz restrictions. More precisely,
the CLSP $\{ {H^{q+2}(\Omega),J_q} \}$ defines an external Rayleigh-Ritz principle for (\ref{eq4-21})-(\ref{eq4-22}).

The pair $\{U^h,J_q\}$ is called \emph{discrete least-squares principle} (DLSP) where
${U^h} \subset H^{q+2}(\Omega)$ and $J_q$ is given by (\ref{eq4-26}). Although
the CLSP  describes a mathematically well-posed variational setting, its associated DLSP $\{ {{U^h},J_q} \}$ may
describe an algorithmically infeasible setting. For instance, the least-squares functional may contain inner products in
fractional-order Sobolev spaces on the boundary that are inconvenient for actual implementations.
Practical issues may force us to abandon the DLSP setting
described above and consider instead another pair for DLSP, denoted by $\{ {{U^h},J^h} \}$, where
$$
J^h(u) = \frac{1}{2}\big(h^{-t}\|Lu\|_{\tilde V}^2 + h^{-s}\|Bu\|_{\tilde W}^2\big)
$$
with proper (computationally feasible) Sobolev spaces  $\tilde V$ and $\tilde W$ and nonnegative powers $t$ and $s$ \cite{AZIZ,PB2,PB3,PB5}. The pair $\{ {{U^h},J^h} \}$ is called the {\em data-weighted DLSP}.
Also, a weighted least-squares strong-form based on RBF collocation is given in \cite{cheung-ling-schaback:2018}.

\section{RBFs discretization and error estimation}\label{sect-discretization}

Up to here we are given a least-squares functional which
is equivalent to a combination of Sobolev norms, but these norms might be inconvenient from
the computational point of view. To circumvent this flaw, when this functional is restricted to a finite subspace, we can use the
fact that all norms on a finite-dimensional space are equivalent. Thus, essentially all
norms can be replaced by ${L^2}$-norms weighted by some respective equivalence constants.
In this section we try to introduce a mesh-dependent 
least-squares functional
by using RBFs
where the residual of each equation is measured in the ${L^2}$-norm multiplied by
a weight determined by the equation index and the mesh parameter $h$.
As some earlier work for weighted least-squares methods 
we refer the reader to \cite{AZIZ,PB5}.
\\
\subsection{Weighted discretization of CLSP by RBFs}
Recall the kernel $\Phi$ satisfying \eqref{eq1_12} for
$\tau\ge q+2 >d/2$ to form the data dependent trial space $U_{\Phi,X}$
for a quasi-uniform set $X$. Throughout the paper, $\tau$ (the smoothness index of $\Phi$) satisfies $\tau>d/2$ and it is fixed.
Assuming $h=h_{X,\Omega}$, we define the convex data-weighted functional
\begin{equation}\label{eq5-1}
J^h(u;f,g) := \frac{1}{2}\left(\| Lu - f \|_{0,\Omega }^2 + h^{- 3}\|u - g\|_{0,\partial \Omega }^2 \right),\quad u\in {H^{q + 2}}(\Omega ),
\end{equation}
for $q\geq 0$.
The corresponding
data-weighted DLSP $\{ {U_{\Phi ,X},{J^h}}\}$ then leads to the unconstrained minimization problem
\begin{equation}\label{eq5-2}
\mbox{seek}\; {u^h} \in U_{\Phi ,X} \; \mbox{such that}\; J^h({u^h};f,g) \le J^h({v^h};f,g)\quad \forall \,{v^h} \in U_{\Phi ,X}.
\end{equation}
The Euler-Lagrange equation for \eqref{eq5-2} is given by the variational problem
\begin{equation}\label{eq5-3}
\mbox{seek}\;  {u^h} \in U_{\Phi ,X}\; \mbox{such that}\; Q^h(u^h,v^h) = F^h(v^h)\quad \forall \,{v^h} \in U_{\Phi ,X},
\end{equation}
where
\begin{equation*}\label{eq5-4}
Q^h(u,v) = {({{L}}u,{{L}}v)_{0,\Omega }} + {h^{ - 3}}{(u,v)_{0,\partial \Omega }},
\end{equation*}
and
\begin{equation*}\label{eq5-5}
F^h(v) = {({{L}}v,f)_{0,\Omega }} + {h^{  - 3}}{(v,g)_{0,\partial \Omega }}.
\end{equation*}
The bilinear form $Q^h(\cdot,\cdot)$ defines an inner product
$
((\,\cdot\,,\,\cdot\,)):U_{\Phi ,X} \times U_{\Phi ,X} \to \mathbb{R}
$
by
$$
(({u},{v})) := Q^h({u},{v})
$$
which is called the \emph{data-weighted discrete energy inner product}.
The \emph{data-weighted discrete energy norm} is then defined via
$|||{u}||| := \sqrt {(({u},{u}))}  = \sqrt {2{J^h}({u})}$.
Moreover, the discrete energy inner product and norm can be extended to all functions in $H^{q+2}(\Omega)$.

If we define $\phi_j = \Phi(\cdot-x_j)$ and $A_{kj} = Q^h(\phi_k,\phi_j)$ and $b_k = F^h(\phi_k)$ then the final linear system $Ac = b$
gives the solution vector $c$ for $u^h = \sum_{j=1}^N c_j \Phi(\cdot-x_j)$.

\subsection{Error Estimates}\label{sect-error-estimation}
First of all, 
we are interested in finding elements from $U_{\Phi,X} \subset H^{q+2}(\Omega)$ which are closest
to $u\in H^{q+2}(\Omega)$ .
 More precisely, we are interested in the minimization problem
\begin{equation*}\label{eq-inf1}
\mathop {\inf }\limits_{{v} \in U_{\Phi ,X}} |||u - {v}|||.
\end{equation*}
Since $U_{\Phi,X}$ is convex and $J^h$ is a strictly convex functional then from
\cite[Theorem 38.C]{Zeidler}, $J^h$ has at most one minimum on $U_{\Phi,X}$.
On the other hand, (\ref{eq5-1}) guarantees that
the discrete energy norm can be extended to all smooth functions $u \in {H^{q + 2}}(\Omega )$ for $q \ge 0$, and 
since $U_{\Phi,X}\subset H^{q+2}(\Omega)$, we have
\begin{equation*}\label{eq5-101}
((u,{v^h})) = F^h({v^h}),\quad \forall {v^h} \in U_{\Phi ,X}.
\end{equation*}
From \eqref{eq5-3} we also have $((u^h,{v^h})) = F^h({v^h})$ for all ${v^h} \in U_{\Phi ,X}$.
Subtraction gives
\begin{equation}\label{eq5-11}
((u - {u^h},{v^h})) = 0,\quad \forall {v^h} \in U_{\Phi ,X}.
\end{equation}
Therefore, for all ${v^h} \in U_{\Phi ,X}$ we can write
\begin{align*}
{|||u - {u^h}|||^2} &= ((u - {u^h},u - {u^h}))\, \\
&=((u - {u^h},u - {v^h} + {v^h} - {u^h}))\\
& = ((u - {u^h},u - {v^h}))\\
& \le (|||u - {u^h}|||)(|||u - {v^h}|||).
\end{align*}
This leads that the minimizer of the data-dependent DLSP $\{ {U_{\Phi ,X},{J^h}(\,\cdot\,)} \}$
is the best approximation 
of the minimizer of CLSP $\left\{ {{H^{q + 2}}(\Omega ),{J_q}(\,\cdot\,)} \right\}$ 
out of subspace $U_{\Phi ,X}$ in the discrete energy norm.
\begin{The}\label{th4-7}
Let $q\ge 0$ be given and
$u \in {H^{ q+2}}(\Omega )$ be the unique solution of (\ref{eq4-21})-(\ref{eq4-22}).
Then, $\{ {U_{\Phi ,X},{J^h}(\,\cdot\,)} \}$ has at most one minimizer ${u^h} \in U_{\Phi ,X}$. Also the
minimizer ${u^h} \in U_{\Phi ,X}$ is the orthogonal projection of $u$ with respect to the discrete energy norm; i.e.,
\begin{equation}\label{eq5-10}
\mathop {\inf }\limits_{{v^h} \in U_{\Phi ,X}} |||u - {v^h}||| = |||u - {u^h}|||.
\end{equation}
\end{The}
%
Note that the orthogonality (\ref{eq5-11}) yields the Pythagorean law
${{|||u - {u^h}|||} ^2} + {{|||{u^h}|||} ^2} = { {|||u|||} ^2}$
giving immediately the stability bounds
$|||u - {u^h}||| \le |||u|||$
and $|||{u^h}||| \le |||u|||$ in the discrete energy norm.

We give the error analysis of the method in two parts. In the first part some segment of the error bound can be obtained
from the analysis of least squares methods given in \cite{AZIZ}.
The approximation space $U_{\Phi,X}$ should possess an optimality property with respect to some pairs $(r,k)$.
This property is addressed in the following lemma
which is a direct consequence of Theorem \ref{th2_1} and a result on approximation in scales of Banach spaces \cite{Br3}.
\begin{Lem}\label{lem*2-1}
Under the assumptions on $X$, $\Omega$, $\Phi$, and ${h}<1$ made in Theorem \ref{th2_1}, with
$\tau  \ge k >d/2$,
for all $u \in {H^k }(\Omega )$ 
there exist a function
$v^h \in {U_{\Phi ,X}}$ and a constant $C > 0$ independent of $h$ and $u$ such that for all $0 \le r \le k$
\begin{equation}\label{eq1_188}
\mathop {\inf }_{v^h \in {U_{\Phi ,X}}} \,\sum_{i = 0}^r {h^i{{\| {u - v^h} \|}_{i,\Omega }}}  \le C h^k{\left\| u \right\|_{k,\Omega }}.
\end{equation}
\end{Lem}
\begin{Lem}\label{lem4-8}
Assume $u \in {H^k}(\Omega )$ is given as the unique minimizer of
$\left\{ {{H^{q + 2}}(\Omega ),{J_q}(\,\cdot\,)} \right\}$ and ${u^h} \in U_{\Phi ,X}$ indicates the unique minimizer of
$\{ {U_{\Phi ,X},{J^h}(\,\cdot\,)} \}$ for some $q \ge 0$.
Then, there exists a constant $C>0$ such that for all $s$ with
$\tau  \ge k \ge s \ge q+2> d/2$
 we have
\begin{equation*}\label{eq5-12}
|||u - {u^h}||| \le C\,{h^{s  - 2}}\,{\left\| u \right\|_{s,\Omega }}.
\end{equation*}
\end{Lem}
\begin{proof}
 Let ${v^h} \in U_{\Phi ,X}$.  Using the definition of $|||\,\cdot\,|||$ we have
\begin{align*}
|||u - {v^h}||| &= {\left( {{}\| {{{L}}(u - {v^h})} \|_{0,\Omega }^2 + {h^{  - 3}}\| {u - {v^h}} \|_{0,\partial \Omega }^2} \right)^{1/2}}\\
& \le C\left( {{{\| {u - {v^h}} \|}_{2,\Omega }} + {h^{  - 3/2}}{{\| {u - {v^h}} \|}_{0,\partial \Omega }}} \right)\\
& \le C\left( {{{\| {u - {v^h}} \|}_{2,\Omega }} + {h^{ - 1}}{{\| {u - {v^h}} \|}_{1,\Omega }} + {h^{ - 2}}{{\| {u - {v^h}} \|}_{0,\Omega }}} \right),
\end{align*}
where in the last inequality , for $0 < h \le 1$ the inequality (see \cite{Br4,hangelbroek-et-al:2017})
\begin{equation}\label{L2boundH1L2}
\|v\|_{0,\partial\Omega}\leqslant C(\eta^{-1}\|v\|_{0,\Omega}+\eta \|v\|_{1,\Omega})
\end{equation}
 is used for ${\left\| {u - {v^h}} \right\|_{0,\partial \Omega }}$ with $\eta  = {h^{1/2}}$. Therefore,
$$
\mathop {\inf }_{{v^h} \in U_{\Phi ,X}} |||u - {v^h}||| \le C\,{h^{ - 2}}\,\mathop {\inf }_{{v^h} \in U_{\Phi ,X}} \,
\sum_{i = 0}^2 {{h^i}\,{{\| {u - {v^h}} \|}_{i,\Omega }}} \,.
$$
Using Lemma \ref{lem*2-1} and
(\ref{eq5-10}),
the desired bound is obtained.
\end{proof}
According to Lemma \ref{lem*2-1}, $U_{\Phi,X}$ {\em approximates optimally} with respect to $(r,k)$ for $u\in H^k(\Omega)$ and $0\leqslant r\leqslant k$ in the sense of
\cite{Br3,AZIZ}. Thus, we may do some modifications to the statement and the proof of \cite[Theorem 4.1]{AZIZ} to tune the following result for the kernel-based least-squares method.
\begin{The}\label{th4-9}
Assume $u \in {H^k}(\Omega )$ is given as the unique minimizer of
$\left\{ {{H^{q + 2}}(\Omega ),{J_q}(\,\cdot\,)} \right\}$ and ${u^h} \in U_{\Phi ,X}$ indicates the unique minimizer of
$\{ {U_{\Phi ,X},{J^h}(\,\cdot\,)} \}$ for some $q \ge 0$.
 There exists a constant $C>0$ such that for all $s$ with
$\tau  \ge k \ge s \ge q+2> d/2$ we have
$$
\begin{array}{lll}
\| {{{L}}(u - {u^h})} \|_{ t, \Omega} & \le C\,{h^{s-t -2}}{\left\| u \right\|_{s,\Omega }},&  2-k \le t \le 0,\\
\| {u - {u^h}} \|_{t,\partial \Omega }&  \le C\,{h^{s - t - 1/2}}\,{\left\| u \right\|_{s,\Omega }},& \frac{1}{2}-k \le t \le   0.
\end{array}
$$
\end{The}

\begin{proof}
 Let ${f_1} \in {H^{p-2}}(\Omega )$ and
${g_1} \in {H^{p - 1/2}}(\partial \Omega )$
be given for a fixed $p$ such that $\tau \ge k  \ge p > d/2$ and $p \ge q + 2$.
From Theorem \ref{th4_3}, there exists a function $\varphi  \in {H^{p}}(\Omega )$ that satisfies
\begin{equation}\label{eq5-131}
\begin{cases}
L\varphi  = {f_1}, & \mbox{in}\,\,\,\Omega ,\\
\varphi  = {g_1},& \mbox{on}\,\,\,\partial \Omega .
\end{cases}
\end{equation}
For ${v^h} \in U_{\Phi ,X}$, from (\ref{eq5-11}) and the Cauchy-Schwarz inequality we obtain
\begin{align*}
((u - {u^h},\,\varphi )) &= ((u - {u^h},\,\varphi  - {v^h}))\\
&\le (|||u - {u^h}|||)(|||\varphi  - {v^h}|||).
\end{align*}
Lemma \ref{lem4-8} then yields
\begin{align*}
((u - {u^h},\,\varphi )) &\le (|||u - {u^h}|||)(\mathop {\inf }\limits_{{v^h} \in U_{\Phi ,X}} |||\varphi  - {v^h}|||)\\
& \le C({h^{s  - 2}}{\left\| u \right\|_{s,\Omega }})({h^{p  - 2}}{\left\| \varphi \right\|_{p,\Omega }})\\
& = C{h^{s + p  - 4}}{\left\| u \right\|_{s,\Omega }}{\left\| \varphi \right\|_{p,\Omega }},
\end{align*}
for all $s$ with $\tau \ge k  \ge s > d/2$.
 Now, we can apply Lemma \ref{cor4_1} for problem (\ref{eq5-131}) to obtain
$$
((u - {u^h},\varphi)) \le C{h^{s + p  - 4}}{\left\| u \right\|_{s,\Omega }}\left\{ {{{\left\| {{f_1}} \right\|}_{p - 2,\Omega }} + {{\left\| g_1 \right\|}_{p- 1/2,\partial \Omega }}} \right\}.
$$
From the definition of the discrete energy inner product we have 
\begin{equation}\label{eq5-14}
{}{({{L}}(u - {u^h}),{f_1})_{0,\Omega }} + {h^{ - 3}}{(u - {u^h},{g_1})_{0,\partial \Omega }} \le C{h^{s + p  - 4}}{\left\| u \right\|_{s,\Omega }}\left\{ {{{\left\| {{f_1}} \right\|}_{p - 2,\Omega }} + {{\left\| g_1 \right\|}_{p - 1/2,\partial \Omega }}} \right\}.
\end{equation}
In particular, let $g_1=0$ in (\ref{eq5-14}) to get 
$$
{({{L}}(u - {u^h}),{f_1})_{0,\Omega }} \le C{h^{s + p - 4}}{\left\| u \right\|_{s,\Omega }}{\left\| {{f_1}} \right\|_{p - 2,\Omega }}.
$$
Consequently,
\begin{equation}\label{eq5-15}
{\| {{{L}}(u - {u^h})} \|_{ - (p - 2),\Omega }}  = \mathop {\sup }\limits_{{f_1} \in {H^{p - 2}}(\Omega )}  \frac{{{{({{L}}(u - {u^h}),{f_1})}_{0,\Omega }}}}{{{{\left\| {{f_1}} \right\|}_{p - 2,\Omega }}}} \le C{h^{s + p - 4}}{\left\| u \right\|_{s,\Omega }}.
\end{equation}
In particular, for $p=k$ in (\ref{eq5-15}) we obtain
\begin{equation}\label{eq5-16}
{\| {{{L}}(u - {u^h})} \|_{- (k - 2),\Omega }} \le C\,{h^{s + k - 4}}\,{\left\| u \right\|_{s,\Omega }}.
\end{equation}
Also, the definition of the discrete energy norm 
implies
$$
{\| {{{L}}(u - {u^h})} \|_{0,\Omega }} \le \,|||u - {u^h}|||,
$$
leading to
\begin{equation}\label{eq5-17}
{\| {{{L}}(u - {u^h})} \|_{0,\Omega }} \le \,C\,{h^{s - 2}}\,{\left\| u \right\|_{s,\Omega }},
\end{equation}
by applying Lemma \ref{lem4-8}. Bounds \eqref{eq5-16} and \eqref{eq5-17}
give the estimates for $t = 0$ and $t=-(k-2)$, respectively.  We can use the interpolation theorem on Sobolev spaces
(see \cite[Chapter 4]{Adam} or
\cite[Proposition 2.3]{Lions1})
 to get
\begin{equation}\label{eq5-18}
{\| {{{L}}(u - {u^h})} \|_{t,\Omega }} \le C{\left( {{{\| {{{L}}(u - {u^h})} \|}_{ - (k - 2),\Omega }}} \right)^\theta }{\left( {{{\| {{{L}}(u - {u^h})} \|}_{0,\Omega }}} \right)^{1 - \theta }},
\end{equation}
where $t =  - (k - 2)\theta $, for $0 \le \theta  \le 1$. Inserting estimates (\ref{eq5-16}) and (\ref{eq5-17}) into (\ref{eq5-18})
yields
$$
\| {{{L}}(u - {u^h})} \|_{ t, \Omega}  \le C\,{h^{s-t -2}}{\left\| u \right\|_{s,\Omega }},\quad  2-k \le t \le 0,
$$
for $2 - k \le t \le 0$.
\\
Now, it is possible to choose ${f_1} = 0$ in (\ref{eq5-14}) to get
\[
{(u - {u^h},{g_1})_{0,\partial \Omega }} \le C\,{h^{s + p - 1}}{\left\| u \right\|_{s,\Omega }}{\left\| {{g_1}} \right\|_{p -1/2,\partial \Omega }}.
\]
Hence,
\begin{equation}\label{eq5-20}
{\| {u - {u^h}} \|_{ - (p - 1/2),\partial \Omega }} \le C\,{h^{s + p - 1}}\,{\left\| u \right\|_{s,\Omega }}.
\end{equation}
Now, let $p=k$ in (\ref{eq5-20}) to obtain
\begin{equation}\label{eq5-21}
{\| {u - {u^h}} \|_{ - (k - 1/2),\partial \Omega }} \le C\,{h^{s + p - 1}}\,{\left\| u \right\|_{s,\Omega }}.
\end{equation}
Furthermore, from the definition of the discrete energy norm 
we have
\[
{\| {u - {u^h}} \|_{0,\partial \Omega }} \le \,{h^{ 3/2}}|||u - {u^h}|||.
\]
leading to
\begin{equation}\label{eq5-22}
{\| {u - {u^h}} \|_{0,\partial \Omega }} \le C\,{h^{s - 1/2}}\,{\| u \|_{s,\Omega }}
\end{equation}
after applying Lemma \ref{lem4-8}.
Now, using the interpolation theorem on trace Sobolev spaces \cite{Adam,Lions1} we have
\begin{equation}\label{eq5-23}
{\| {u - {u^h}} \|_{t + 3/2,\partial \Omega }} \le C{({\| {u - {u^h}} \|_{ - (k - 1/2),\partial \Omega }})^\theta }{({\| {u - {u^h}} \|_{0,\partial \Omega }})^{1 - \theta }},\quad 0 \le \theta  \le 1,
\end{equation}
where
$t + 3/2 =  - (k - 1/2)\theta $.
Inserting (\ref{eq5-21}) and (\ref{eq5-22}) to (\ref{eq5-23})
gives
\[
{\| {u-{u^h}}\|_{t + 3/2,\partial \Omega }} \le C\,{h^{s-t-2}}\,{\left\| u \right\|_{s,\Omega }},\quad - k - 1 \le t \le   - {\textstyle{3 \over 2}}.
\]
\end{proof}
\begin{The}\label{th4-11}
Assume that $u \in {H^k}(\Omega )$ is the unique solution of (\ref{eq4-21})-(\ref{eq4-22})
in the CLSP $\left\{ {{H^{q + 2}}(\Omega ),{J_q}(\,\cdot\,)} \right\}$
such that
$\tau  \ge k > d/2$
and $k \ge \max \{ q + 2,4\}$ for some real $q \ge 0$. Also, assume that ${u^h} \in U_{\Phi ,X}$ is the unique minimizer of
$\{ {U_{\Phi ,X},{J^h}(\,\cdot\,)} \}$.
Then, there exists a constant $C>0$ such that for all $s$ with
$ k \ge s \ge q+2 > d/2$ we have
\begin{equation*}\label{eq5-24}
{\| {u - {u^h}} \|_{t,\Omega }} \le C\,{h^{s - t}}\,{\left\| u \right\|_{s,\Omega }},\,\,\,\,\,\,\,\,\,\,\,\,\,\,\,\,\,\,\,\,\,\,\,\,4 - k \le t \le
{\textstyle{1 \over 2}}.
\end{equation*}
\end{The}
\begin{proof}
 Using Lemma \ref{cor4_1} we have
\[
{\| {u - {u^h}} \|_{t + 2,\Omega }} \le C\left( {{{\| {{{L}}(u - {u^h})} \|}_{t,\Omega }} + {{\| {u - {u^h}} \|}_{t + 3/2,\partial \Omega }}} \right),
\]
for all real $t \le {k-2}$. Hence, Theorem \ref{th4-9}
yields
\[
{\| {u - {u^h}} \|_{t + 2,\Omega }} \le C\,\,{h^{s - t - 2}}\,{\| u \|_{s,\Omega }},\quad 2 - k \le t \le  -
{\textstyle{3 \over 2}},
\]
or equivalently
\[
{\| {u - {u^h}} \|_{t,\Omega }} \le C\,\,{h^{s - t}}\,{\left\| u \right\|_{s,\Omega }},\quad 4 - k \le t \le
{\textstyle{1 \over 2}}.
\]
\end{proof}

The last bound of Theorem \ref{th4-11} measures the error in Sobolev norms
$\|\cdot\|_{t,\Omega}$ for $4-k\le t\le 1/2$ that also includes the error estimation in the $L_2$ norm because $k\geq 4$.
The reminder parts of this section are devoted to prove a norm-equivalent property and to extend the above error analysis to
higher order Sobolev norms on the left hand side.
However, our results are hanged on an {\em open problem} that will be stated after some auxiliary lemmas form the
kernel approximation theory. 

The proof of the following {\em inverse inequality} of Bernstein type can be found in \cite{cheung-ling-schaback:2018}.
\begin{Lem}\label{th2_2}
Assume a kernel $\Phi$ satisfying (\ref{eq1_12}) with
 $\tau  > d/2$ is given.
Suppose the domain $\Omega$ is a bounded Lipschitz region satisfying an interior cone condition.
Then for all $u^h \in {U_{\Phi ,X}}$ and all  finite sets $X = \left\{ {{x_1},\ldots,{x_N}} \right\} \subset \Omega $
with separation distance ${q_X}$, there is a constant $C$ depending only on $\Phi$, $\Omega$ and $\mu$ such that for all
$d/2 < \mu  \le \tau $ we have
\begin{equation}\label{eq1_20}
{\| u^h \|_{\tau ,\Omega }} \le C\,q_{X}^{ -\tau + \mu }{\| u^h \|_{\mu  ,\Omega }}.
\end{equation}
\end{Lem}
By applying Theorem \ref{th4_3}, Lemma \ref{th2_2} with $\mu  =2$ and for quasi-uniform sets $X$ (i.e. $q_X\approx h_{X,\Omega}$), and
by invoking an interpolation argument we can prove that
for all ${u^h} \in U_{\Phi ,X}$ there exists a constant $C>0$, independent of $u^h$, such that for all $q$
with $0 \le q  \le \tau -2 $,
\begin{equation}\label{eq5-52}
{\| {{{L}}{u^h}} \|_{q,\Omega }} \le C{h_{X,\Omega}^{ - q}}\left({\| {{L}{u^h}} \|_{0,\Omega }} + {\| {{u^h}} \|_{3/2,\partial \Omega }}\right)
\end{equation}
where $\Omega  \subset {\mathbb{R}^d}$ with $d \le 3$.

We also need a
{\em sampling inequality} or {\em zeros lemma} to support our argument. 
A variation of zeros lemma that holds for fractional Sobolev norms on both sides of inequality has been proved in \cite{HW4}.
Also, see \cite{NAR2} for older versions.
\begin{Lem}\label{Lem-Sample}
Suppose $\Omega \subset \mathbb{R}^d$ is a bounded Lipschitz domain.
Let $r,k \in \mathbb{R}$ satisfy
$k > d/2$ and $0 \le r \le k$.
If $u \in {H^k}(\Omega )$ satisfies $u\left| {_X} \right. = 0$, then for any discrete sets $X \subset \Omega$ with
sufficiently small mesh norm $h_{X,\Omega}$, there exists a constant $C$ that depends only on $\Omega$ and $k$ such that
\[\,{\left\| u \right\|_r} \le Ch_{X,\Omega}^{k - r}\,{\left\| u \right\|_k}, \quad 0 \le r \le k.\]
\end{Lem}

The weighted DLSP setting uses the $L^2$-norm for the boundary part while CLSP involves the boundary norm
$\left\|\cdot \right\|_{q + 3/2,\partial \Omega }$.
Thus, we need an inverse inequality that relates
$\left\| \cdot \right\|_{3/2,\partial \Omega }$
and
$\left\| \cdot \right\|_{0,\partial \Omega }$ for approximating function $u^h \in U_{\Phi,X}$.
The proof of
such boundary inverse inequality seems to need a counterpart inverse inequality in $\Omega$ that is still an open problem. Thus we conjecture
\begin{itemize}
\item[]\textbf{Conjecture A.} For all finite quasi-uniform set
$X \subset \Omega \subset \mathbb{R}^d$, with sufficiently small fill distance $h$,
 there exists a constant $C$, depending only on $\Omega$, $\partial \Omega$, and $\Phi$ such that for all $u \in {U_{\Phi ,X}}$,
\begin{equation}\label{eq1_231}
\| u \|_{3/2,\partial \Omega } \le C\,h^{ - 3/2}\| u \|_{0,\partial \Omega }.
\end{equation}
\end{itemize}
Apart form Theorem \ref{th4-11}, Conjecture A will be used to prove the error bound for the least-squares method in $\|\cdot\|_{t,\Omega}$ for $0< t\le k$ when $u\in H^k(\Omega)$; see Theorem  \ref{thm4-11} below. Also, it supports our estimation for the lower bound of the smallest eigenvalue of the final matrix in section \ref{sect-conditioning}.

Using Conjecture A, in the following we show that the weighted least-squares
functional satisfies a data-dependent energy balance.
\begin{Lem}\label{th4-5}
For all real $q$ with $0 \le q \le \tau  - 2$, there exists a positive constant $C$ independent of $u^h$ such that for all
${u^h} \in U_{\Phi ,X}$ with  $0 < h \le 1$ and $d \le 3$ the inequality
\begin{equation}\label{eq5-53}
C^{-1}h^{2q}\| u^h \|_{q + 2,\Omega }^2 \le Q^h({u^h},{u^h})
\end{equation}
holds.
\end{Lem}
\begin{proof}
Let $q=\tau -2$. Using Lemma \ref{cor4_1} and bound \eqref{eq5-52} we have
$$
\begin{array}{rll}
\| {{u^h}} \|_{\tau,\Omega }^2 &  \le C\,{\left( {{{\| {{{L}}{u^h}} \|}_{\tau -2,\Omega }} + {{\| {{u^h}} \|}_{\tau- 1/2,\partial \Omega }}} \right)^2} &\\
&\le C\,{\left( {{{\| {{{L}}{u^h}} \|}_{\tau -2,\Omega }} + {{\| {{u^h}} \|}_{\tau ,\Omega }}} \right)^2}&(\mbox{Theorem \ref{lem4_2}}) \\
&\le C\,{h^{ - 2(\tau -2)}}\,{\left( {{{\| {{{L}}{u^h}} \|}_{0,\Omega }} + {{\| {{u^h}} \|}_{3/2,\partial \Omega }} + {{\| {{u^h}} \|}_{2,\Omega }}} \right)^2}& (\mbox{using \eqref{eq5-52} and Lemma \ref{th2_2}})\\
& \le C\,{h^{ - 2(\tau -2)}}\,{\left( {{{\| {{{L}}{u^h}} \|}_{0,\Omega }} + {{\| {{u^h}} \|}_{3/2,\partial \Omega }}} \right)^2}&(\mbox{using (\ref{eq4-321})})\\
&\le C\,{h^{ - 2(\tau -2)}}\,\left( {\| {{{L}}{u^h}} \|_{0,\Omega }^2 + \| {{u^h}} \|_{3/2,\partial \Omega }^2} \right)\\
&= C\,{h^{ - 2(\tau -2)}}\,\left( {\| {{{L}}{u^h}} \|_{0,\Omega }^2 + h^{-3}\| {{u^h}} \|_{0,\partial \Omega }^2} \right)& (\mbox{using (\ref{eq1_231}}))\\
& = C\,{h^{ - 2(\tau -2)}} Q^h({u^h},{u^h}).
\end{array}
$$
Now, by using the interpolation theorem for $q=0$ and $q = \tau-2$ we get
\begin{align*}
{\| {{u^h}} \|_{q + 2,\Omega }} &  \le C{\left( {{{\| {{u^h}} \|}_{\tau ,\Omega }}} \right)^{1 - \theta }}{\left( {{{\| {{u^h}} \|}_{2,\Omega }}} \right)^\theta }\\
& \le C{\left( {{h^{2 - \tau }}\left[ {{{\| {{{L}}{u^h}} \|}_{0,\Omega }} + {{\| {{u^h}} \|}_{3/2,\partial \Omega }}} \right]} \right)^{1 - \frac{{2 - \tau  + q}}{{2 - \tau }}}}{\left( {{{\| {{{L}}{u^h}} \|}_{0,\Omega }} + {{\| {{u^h}} \|}_{3/2,\partial \Omega }}} \right)^{\frac{{2 - \tau  + q}}{{2 - \tau }}}} \\
& \le C{h^{ - q}}({\| {{{L}}{u^h}} \|_{0,\Omega }} + {\| {{u^h}} \|_{3/2,\partial \Omega }}),
\end{align*}
where $q + 2 = (1 - \theta )\tau  + 2\theta $, for $0 \le \theta  \le 1$. Hence, for all $0\, \le q \le \tau  - 2$,
$$
\begin{array}{rll}
\| {{u^h}} \|_{q + 2,\Omega }^2 &\le C{h^{ - 2q}}(\| {{{L}}{u^h}} \|_{0,\Omega }^2 + \| {{u^h}} \|_{3/2,\partial \Omega }^2) &\\
&  \le C{h^{ - 2q}}(\| {{{L}}{u^h}} \|_{0,\Omega }^2 + {h^{ - 3}}\| {{u^h}} \|_{0,\partial \Omega }^2)&
(\mbox{using (\ref{eq1_231}}))\\
&= C {h^{ - 2q}}Q^h( {{u^h},{u^h}} ) &.
\end{array}
$$
\end{proof}
\begin{Lem}\label{th*4-5}
There exist a positive constant $C$, independent of $u^h$, such that for all
${u^h} \in U_{\Phi ,X}$ with  $0 < h \le 1$ we have
\begin{equation*}\label{eq5-531}
Q^h({u^h},{u^h}) \le C\,{h^{ - 4}}\| u^h \|_{ 2,\Omega }^2.
\end{equation*}
\end{Lem}
\begin{proof}
From (\ref{eq4-321}) for some constant $C>0$ we get
\begin{align*}
Q^h({u^h},{u^h}) =& \| {{{L}}{u^h}} \|_{0,\Omega }^2 + {h^{ - 3}}\| {{u^h}} \|_{0,\partial \Omega }^2 \\
&\le {\left( {{{\| {{{L}}{u^h}} \|}_{0,\Omega }} + {h^{ -3/2}}{{\| {{u^h}} \|}_{0,\partial \Omega }}} \right)^2}\\
& \le C{\left( {{{\| {{u^h}} \|}_{2,\Omega }} + {h^{ -3/2}}{{\| {{u^h}} \|}_{0,\partial \Omega }}} \right)^2}.
\end{align*}
Moreover, for $0 < h \le 1$ by using the inequality (\ref{L2boundH1L2}), 
we obtain for $\eta  = {h^{1/2}}$
\begin{align*}\label{eq5-6}
Q^h({u^h},{u^h}) &\le C{\left\{ {{{\| {{u^h}} \|}_{2,\Omega }} + {h^{- 3/2}}\left( {{h^{ - 1/2}}{{\| {{u^h}} \|}_{0,\Omega }} + {h^{1/2}}{{\| {{u^h}} \|}_{1,\Omega }}} \right)} \right\}^2}\\
&= C{\left( {{}{{\| {{u^h}} \|}_{2,\Omega }} + {h^{  - 1}}{{\| {{u^h}} \|}_{1,\Omega }} + {h^{ - 2}}{{\| {{u^h}} \|}_{0,\Omega }}} \right)^2}\\
& \le C{h^{ - 4}}\| {{u^h}} \|_{2,\Omega }^2.
\end{align*}
\end{proof}
Lemmas \ref{th4-5} and \ref{th*4-5} show that
the discrete energy inner product
$
((\,\cdot\,,\,\cdot\,)):U_{\Phi ,X} \times U_{\Phi ,X} \to \mathbb{R}
$
and the discrete energy norm
$|||\cdot|||$
are related to the Sobolev norms of the solution space. 
\begin{Cor}\label{lem4-6}
Assume  $0 \le q \le \tau  - 2$ is given and  $0 < h \le 1$ is sufficiently small.
 Then, for all ${u^h} \in U_{\Phi ,X}$  there exists a constant $C>0$,
independent of $u^h$, such that
\begin{equation*}\label{eq5-9}
{C^{ - 1}} h^{q}|||{u^h}||{|_q} \le |||{u^h}|||\le C\,{h^{ - 2}}|||{u^h}||{|_0}\le C\,{h^{ - 2}}|||{u^h}||{|_q}.
\end{equation*}
\end{Cor}

In the following, we aim to extend our error analysis in Theorem \ref{th4-11} to other (positive) Sobolev norms.
\begin{Lem}\label{lem4-141}
Suppose that $\Phi$ satisfies (\ref{eq1_12}) with
$\tau  \ge k > d/2$,
 and $\Omega  \subset {\mathbb{R}^d}$ is a bounded Lipschitz domain. Furthermore, let
$X \subset \Omega$ be a discrete set of centers with sufficiently small mesh norm ${h}=h_{X,\Omega}$.
If $u \in {H^k}(\Omega )$ and
the RBF interpolant of $u$ on $X$ is given by ${I_X}u$
then there exists a constant $C>0$ such that for all $s$ with
$\tau  \ge k \ge s \ge q+2 > d/2$
we have
\begin{equation*}\label{eq5-102}
|||u - {I_X}u||| \le C\,{h^{s  - 2}}\,{\left\| {u - {I_X}u} \right\|_s}.
\end{equation*}
\end{Lem}
\begin{proof}
Using the definition of $|||\,\cdot\,|||$ we can write
\begin{align*}
|||u - {I_X}u||| &= {\left( {{}\left\| {{{L}}(u - {I_X}u)} \right\|_{0,\Omega }^2 + {h^{  - 3}}\left\| {u - {I_X}u} \right\|_{0,\partial \Omega }^2} \right)^{1/2}}\\
& \le C\left( {{}{{\left\| {u - {I_X}u} \right\|}_{2,\Omega }} + {h^{- 3/2}}{{\left\| {u - {I_X}u} \right\|}_{0,\partial \Omega }}} \right)\\
&\le C\left( {{}{{\left\| {u - {I_X}u} \right\|}_{2,\Omega }} + {h^{  - 1}}{{\left\| {u - {I_X}u} \right\|}_{1,\Omega }} + {h^{  - 2}}{{\left\| {u - {I_X}u} \right\|}_{0,\Omega }}} \right) \\
& \le C({}{h^{s - 2}}\,{\left\| {u - {I_X}u} \right\|_s} + {h^{ - 1}}{h^{s - 1}}\,{\left\| {u - {I_X}u} \right\|_s} + {h^{ - 2}}{h^s}\,{\left\| {u - {I_X}u} \right\|_s}) \\
&=C{h^{s - 2}}\,{\left\| {u - {I_X}u} \right\|_s},
\end{align*}
which completes the proof. In the third line above
the bound \eqref{L2boundH1L2} is applied with
$\eta  = {h^{1/2}}$, and in the fourth line Lemma \ref{Lem-Sample} is used.
\end{proof}
If $u \in {H^k}(\Omega )$ is the unique minimizer of
$\left\{ {{H^{q + 2}}(\Omega ),{J_q}} \right\}$ for $q \ge 0$ and
$u^h\in U_{\Phi,X}$ is the unique minimizer of
$\left\{ U_{\Phi,X},{J^h}\right\}$ then by
using the fact that $u^h$ is the best approximation in the discrete energy norm, and under the assumptions and notations of Lemma \ref{lem4-141}
 we have
\begin{equation*}
|||u - {u^h}||| \le |||u - {I_X}u||| \le C{h^{s - 2}}\,{\left\| {u - {I_X}u} \right\|_{s,\Omega }},
\end{equation*}
which turns the error of PDE solution into the error of pure interpolation problem.
By applying  the bound $\|u-I_Xu\|_{s,\Omega}\leq \|u\|_{s,\Omega}$ from Theorem \ref{th2_1}
we obtain
\begin{equation}\label{bound_energynorm}
|||u - {u^h}||| \le Ch^{s-2}\|u\|_{s,\Omega }, \quad d/2<s\le k\le \tau,
\end{equation}
for true solution $u\in H^k(\Omega)$. Altogether, we have the following theorem.
\begin{The}\label{thm4-11}
Assume that $\Phi$ satisfies \eqref{eq1_12} for $\tau>d/2$ and $u \in {H^k}(\Omega )$ is the unique solution of (\ref{eq4-21})-(\ref{eq4-22}), where $k$
is a real number such that
$\tau  \ge k > d/2$
 for $d \le 3$, and $k \ge \max \{ q + 2,4\}$ for some real $q \ge 0$. 
 Moreover, assume that ${u^h} \in U_{\Phi ,X}$ is the unique minimizer of
$\{ {U_{\Phi ,X},{J^h}} \}$. Then
\begin{equation*}\label{eq5-541}
{\| {u - {u^h}} \|_{t,\Omega }} \le C\,{h^{k - t}}{\left\| u \right\|_{k,\Omega }},\quad 0\le t \le k.
\end{equation*}
\begin{proof}
Let us first assume $t= k$. Then the inequality
\begin{align*}
{\| {u - {u^h}} \|_{k,\Omega }} &\le {\left\| {u - {I_X}u} \right\|_{k,\Omega }} + {\| {{u^h} - {I_X}u} \|_{k,\Omega }},
\end{align*}
suggests that we can focus on the difference  ${u^h} - {I_X}u \in {U_{\Phi ,X}}$
and on the difference ${u} - {I_X}u \in H^k(\Omega)$ where ${I_X}u$ denotes the unique interpolant of the exact solution $u$
from the trial space ${U_{\Phi ,X}} \subset H^\tau(\Omega)$.
For the first norm on the right hand side, from Theorem \ref{th2_1} we have
$$
\left\| {u - {I_X}u} \right\|_{k,\Omega }\le C\|u\|_{k,\Omega}.
$$
Since the result in Theorem \ref{th4-5} only applies to functions in the trial space,  for $d \le 3$,
we obtain
\begin{equation*}
{\| {{u^h} - {I_X}u} \|_{k,\Omega }} \le C {h^{2-k}}|||{u^h} - {I_X}u||| \le C {h^{2-k}}( |||u - {I_X}u||| + |||u - {u^h}|||).
\end{equation*}
Using the discussions right before the theorem for $s=k$, the right hand side can be bounded by $C\|u\|_{k,\Omega}$ to get
$$
\| {{u^h} - {I_X}u} \|_{k,\Omega } \le C \|u\|_{k,\Omega}.
$$
Combining the recent bounds, we obtain the following stability bound
\begin{align*}
{\| {u - {u^h}} \|_{k,\Omega }} \le C{\left\| u \right\|_{k,\Omega }}.
\end{align*}
On the other hand, putting  $s=k$ and $t=0$ in Theorem \ref{th4-11} yields
 \begin{equation*}
{\| {u - {u^h}} \|_{0,\Omega }} \le C\,{h^{ k }}{\left\| u \right\|_{k,\Omega }}.
\end{equation*}
Henceforth, for $0 \le t\le k$, using the interpolation theorem we have
\begin{align*}
{\| {u - {u^h}} \|_{t,\Omega }} &\le C\,{\left( {{{\| {u - {u^h}} \|}_{0,\Omega }}} \right)^{1 - \theta }}{\left( {{{\| {u - {u^h}} \|}_{k,\Omega }}} \right)^\theta }\\
&  \le C{\left( {{h^k}{{\left\| u \right\|}_{k,\Omega }}} \right)^{1 - t/k}}{\left( {{{\left\| u \right\|}_{k,\Omega }}} \right)^{t/k}}\\
& = C\,{h^{k - t}}{\left\| u \right\|_{k,\Omega }},
\end{align*}
where $t = \theta k$ and $0 \le \theta  \le 1$.
\end{proof}
\end{The}
\subsection{Condition Numbers}\label{sect-conditioning}
In this subsection we estimate the condition number of the presented least-squares method.
Under reasonable assumptions, the condition number of the discrete least-squares matrix
is controlled by the mesh size $h$ and regularity parameter $\tau$.
Since the final matrix
$A$ is positive definite,
 its condition number can be defined as
\begin{equation*}
\cond_2(A) = \frac{\lambda_{\max}(A)}{\lambda_{\min}(A)}.
\end{equation*}
where ${{\lambda _{\max }(A)}}$ and ${{\lambda _{\min }(A)}}$ denote the largest and smallest eigenvalue
of %
$A$, respectively.
An appropriate way to bound ${{\lambda _{\min }}}$ is the use of an inverse inequality in the trial
space to turn the conditioning of the PDE 
matrix back to one of the approximation theory.
To this end, we review a lemma from
 \cite{HW3} that computes a lower bound for the
 smallest eigenvalue
 of interpolation matrix by kernel $\Phi$.
\begin{Lem}\label{lem-cond}
Assume $B_{\Phi,X} = (\Phi(x_j-x_k))_{j,k=1}^N$ is the usual interpolation matrix. Then the minimum eigenvalue of $B_{\Phi,X}$ can be bounded by
\begin{equation}\label{lminbound}
\lambda_{\min}(B_{\Phi,X})\geqslant Cq_X^{2\tau-d}.
\end{equation}
\end{Lem}
\begin{Lem}\label{thm-cond1}
Suppose that $\Phi$ satisfies \eqref{eq1_12} for $\tau>d/2$, $X\subset\Omega\subset \mathbb R^d$ is quasi-uniform and $h=h_{X,\Omega}$ is sufficiently small. 
The minimum eigenvalue of $A$, for $d \le 3$, can be bounded by
$$
\lambda_{\min}(A)  \geqslant C h^{4\tau-d-4}.
$$
\end{Lem}
\begin{proof}
 An appropriate formula for $\lambda_{\min}(A)$ is
$$
\lambda_{\min}(A)  = \min_{0\neq \xi\in \R^{N}}  \frac{\xi^TA\xi}{\|\xi\|^2}.
$$
For a given $\xi\in \R^{N}$ assume that
$u^h = \sum_{j=1}^N \xi_j\Phi(\cdot-x_j)$.
Thus, by using (\ref{eq5-53}) and inverse inequality (\ref{eq1_20}) we deduce
\begin{align*}
\xi^TA \xi &= ((u^h,u^h)) = |||u^h|||^2  \\
& \geqslant C h^{2\tau-4}\|u^h\|_{\tau,\Omega}^2 \quad~~~~~~~ (\mbox{ Theorem \ref{th4-5}}) \\\
& \geqslant C h^{2\tau-4}\|u^h\|_{\Phi}^2  \\
& = C h^{2\tau-4}\xi^T B_{\Phi,X}\xi \\
& \geqslant C h^{2\tau-4} h^{2\tau-d}\|\xi\|^2, \quad~~~ (\mbox{using bound \eqref{lminbound}})
\end{align*}
which shows that
$\lambda_{\min}(A)  \geqslant C h^{4\tau-d-4}$.
\end{proof}
To bound ${{\lambda _{\max }(A)}}$,
we need some results about derivatives of positive definite functions. 
In \cite{Bu1, Bu2, Ma} it is proved that certain
derivatives of positive definite functions are also positive (or negative) definite.  Authors show that some simple conditions on
even order derivatives of positive definite functions at the origin strongly determine their global properties. In particular,  they show that the
derivatives of a smooth positive definite function can be estimated in terms of the even order derivatives at the origin.
Proposition 3.2 of \cite{Ma} prove 
that
if $\Phi$ is a positive definite function of class $C^{2n}$ in
some neighborhood of the origin, for some positive integer $n$, then for each $\,\left| \alpha \right| \le n$ the function
$\,{( - 1)^{\left| \alpha \right|\,}}{D^{2\alpha }}\Phi $ is positive definite of class ${C^{2(n - \left| \alpha \right|)}}(\mathbb{R}^d)$.
Also, the following inequality holds for $\,\left| \alpha \right| , \left| \beta \right| \le n$,
\begin{equation*}
|{D^{\alpha  + \beta }}\Phi (x - \cdot)|^2 \le {( - 1)^{|\alpha  + \beta |}}{D^{2\alpha }}\Phi (0){D^{2\beta }}\Phi (0),\quad x \in {\mathbb{R}^d}.
\end{equation*}
Therefore, if $\Phi \in C^{2n}(\Omega)$  for $n \ge 1$, then
\begin{equation*}\label{eq-derivative1}
|{D^\alpha }\Phi (x - \cdot)|^2 \le  - \Phi (0){D^{2\alpha }}\Phi (0)\;\, \mathrm{for} \;\,\left| \alpha \right| = 1,
\end{equation*}
and  for $|\gamma|=2 $
\begin{equation*}\label{eq-derivative2}
|{D^\gamma }\Phi (x - \cdot)|^2 = |{D^{\alpha  + \beta }}\Phi (x - \cdot)|^2 \le {D^{2\alpha }}\Phi (0){D^{2\beta }}\Phi (0)\;\,\mathrm{for} \;\,\left| \alpha \right| = \left| \beta \right| = 1.
\end{equation*}
Henceforth,
\begin{equation}\label{eq-derivative3}
{\sum_{|\alpha | \le 2} {\left\| {{D^\alpha }\Phi (x - \cdot)} \right\|} _{\infty ,\Omega }^2} \le \sum_{|\alpha | = |\beta | = 1} {\max \{ \Phi (0), - \Phi (0){D^{2\alpha }}\Phi (0),{D^{2\alpha }}\Phi (0){D^{2\beta }}\Phi (0)\} } .
\end{equation}
\begin{Lem}\label{thm-cond2}
Suppose that $X\subset\Omega\subset \mathbb R^d$ is quasi-uniform and $h=h_{X,\Omega}$ is sufficiently small. Also, assume that $\Phi \in C^{2n} (\Omega)$
for some $n \ge 1$.
Then the maximum eigenvalue of $A$ can be bounded by
$$
\lambda_{\max}(A) \leqslant C h^{-d-4}.
$$
\end{Lem}
\begin{proof}
We can employ (\ref{eq-derivative3}) and the inequality $|\sum\nolimits_{j = 1}^N {{\xi _j}} {|^2} \le N\sum\nolimits_{j = 1}^N {|{\xi _j}{|^2}} $,  to deduce that
\begin{align*}
\| {{u^h}} \|_{2,\Omega }^2 &= \sum_{\left| \alpha  \right| \le 2} {\int\nolimits_\Omega  {{{[ {{D^\alpha }{u^h}} ]}^2}} } dx\\
&=\sum_{\left| \alpha  \right| \le 2} {\int\nolimits_\Omega  {{{\Big[ {\sum_j {{\xi _j}{D^\alpha }\Phi (x - {x_j})} } \Big]}^2}} } dx\\
& \le \sum_{\left| \alpha  \right| \le 2} {\left\| {{D^\alpha }\Phi (x - {x_j})} \right\|_{\infty ,\Omega }^2 \times \int\nolimits_\Omega  {\Big|\sum_j {{\xi _j}} {\Big|^2}} } dx\\
& \le \sum_{\left| \alpha  \right| \le 2} {\left\| {{D^\alpha }\Phi (x - {x_j})} \right\|_{\infty ,\Omega }^2 \times \Big|\sum_j {{\xi _j}} {\Big|^2} \times \mathrm{vol}(\Omega) } \\
& \le \mathrm{vol}(\Omega) \times \Big|\sum_j {{\xi _j}} {\Big|^2} \times \sum_{\left| \alpha  \right| = \left| \beta  \right| = 1} {\max \{ {\Phi (0),( - 1)\Phi (0){D^{2\alpha }}\Phi (0),{D^{2\alpha }}\Phi (0){D^{2\beta }}\Phi (0)} \}} \\
&  \le C\Big|\sum_j {{\xi _j}} {\Big|^2}
\le CN{\sum_j {\left| {{\xi _j}} \right|} ^2}
\le C h^{-d} \|\xi\|^2. \quad ~~~~~~~(\mbox{Since $N =\mathcal O(h^{-d})$})
\end{align*}
Thus we can conclude
\begin{align*}
\xi^TA \xi &=|||u^h|||^2
\le Ch^{-4}\|u^h\|_{2,\Omega}^2
\le Ch^{-4} h^{-d} \|\xi\|^2
\end{align*}
which gives the desired bound.
\end{proof}
\begin{Cor}\label{cor-condioton}
For $d \le 3$, the condition number of the final linear system of the least-squares kernel-based method is bounded by
$$
\cond_2(A)\leqslant Ch^{-4\tau}.
$$
\end{Cor}
\section{Numerical Examples}\label{sect-numerical}
In this section, results of some numerical experiments are reported to verify the theoretical bounds of the preceding sections.
The convergence of the numerical solution $u^h$ toward the true solution $u$ are investigated and the rates of convergence are estimated numerically.
Besides, the condition number of the final system is
 estimated.
The computational rate of convergence $p$ is approximated in two successive levels $h_1$ and $h_2$ via
\[
p=\log{\frac{{{{\| {u - {u^{{h_1}}}} \|}_{t,\Omega }}}}{{{{\| {u - {u^{{h_2}}}} \|}_{t,\Omega }}}} }\Big/ \log{{\frac{{{h_1}}}{{{h_2}}}}  },
\]
for $t=0,2$. Moreover, to estimate $H^2$ error and convergence rates we use
$$
\|u-u^h\|_{2,\Omega}\approx |||u-u^h|||.   
$$
From Theorems \ref{th4-11} and \ref{thm4-11}, one finds that the theoretical rates are $k-t$ for
 $u \in {H^k}(\Omega )$.

The following test problem in $\R^2$ is considered
\[{{L}}u =  - \frac{{{\partial ^2}u}}{{\partial {(x^1)^2}}} - \frac{{{\partial ^2}u}}{{\partial {(x^2)^2}}} + \frac{{\partial u}}{{\partial x^1}} + \frac{{\partial u}}{{\partial x^2}} + u = f,\quad \mbox{in}\; \Omega. \]
Assume that the exact solution is given by
$
 u^*(x):= {\left\| x \right\|_2^\kappa},
$
over the computational domain
$\Omega  = \left\{ {x \in {\mathbb{R}^2}\,\,:\,\,\,\|x\|_2 < 1} \right\}$.
To form the trial space $U_{\Phi,X}$, we employ the Whittle-Mat\'{e}rn-Sobolev kernel
\[
\Phi (x) = \left(\varepsilon \|x\| \right)^{\tau  - d/2}K_{\tau-d/2}(\varepsilon\| x \|)\quad \mbox{for all} \; x \in {\mathbb{R}^d}, \quad \tau>d/2,
\]
where ${K_\beta }$ is the modified Bessel function of the second kind of order $\beta$, and $\varepsilon>0$ is a shape parameter.
This kernel has the Fourier transform
$\widehat \Phi (\omega)=C{(1 + {\left\|\varepsilon \omega \right\|^2})^{ - \tau }}$ and its native space is identical with $H^\tau(\R^d)$.

Discretization is done by using a series set of trial points $X\subset \Omega\cup\partial \Omega$ with different fill-distance $h$.
However, in this scheme it is not mandatory to locate some trial points on $\partial \Omega$.
We note that, to enforce the boundary conditions in some collocation methods such as the symmetric kernel-based method of \cite{Fra1,Fra2,ZW3}, some trial points have to be qualitatively located on the boundary.
All reported errors in $L^2(\Omega)$ norm 
are RMS errors approximated by using the fixed set of $7668$ equidistant points in $\Omega$. Also,
the error is analyzed in $\|\cdot\|_{0,\partial\Omega}$
with $1000$ equidistant points on $\partial \Omega$. All integrals in DLSP variational form
are computed via the Gauss-Legendre quadrature rule with sufficient number of integration points in angular and radial directions.
We do not employ any special technique to deal with the problem of ill-conditioning.

For a nonnegative integer $k$, it is well-known that
${u^*} \in {H^k}(\Omega )$ if $\kappa  > k - d/2$.
In our numerical example, we set $\kappa=4$ to have ${u^*} \in {H^k}(\Omega )$ for $k < 5 $.
We set $q=0$ in the least-squares approach, $\varepsilon = 10$ for shape parameter, and $\tau= 3,4,5,6$ for kernel function $\Phi$.
From Theorems \ref{th4-11} and \ref{thm4-11} the parameters $\tau$ and $k$ should satisfy $\tau  \ge k \ge 4$. Thus the cases
$\tau=3,4$ exclude the requirements of the theory, but yet perform the $L^2$ convergence; 
see Table \ref{table-1}. For brevity, the notation $e^h=u^*-u^h$ is used.
In cases $\tau=5,6$, the table contains also the theoretical orders of Theorems \ref{th4-11} and \ref{thm4-11} and the bound \eqref{bound_energynorm}.
In all columns, except that of $|||e^h|||$, the numerical orders at finer levels are better than the expected theoretical orders.
Note that the orders do not improve when going from $\tau=5$ to $\tau=6$ because of the limitation caused by the smoothness of true solution $u^*$.

According to the results of subsection \ref{sect-conditioning}, there is a direct relation between the smoothness of trial kernel and the conditioning
of final system, where a higher smoothness leads to a larger condition number.
Since in this paper we do not focus on preconditioning techniques, for small values of $h$ the results suffer from a severe ill-conditioning, specially for higher values $\tau=5,6$. This is the reason why fewer rows are reported in this cases.
Table \ref{table-2} shows  the condition numbers of the final linear systems together
with the numerical orders.
In all cases, as $h \to 0$, the approximate rate of
conditioning of the final matrix is of $\mathcal O({h^{ - 4\tau}})$ as proven
in Corollary \ref{cor-condioton}. Unsatisfactory results for small values of $h$ and higher values of $\tau$ are obtained.

\begin{table}[]
\begin{center}
\caption{\small{Approximate errors and orders for DLSP using Whittle-Mat\'{e}rn-Sobolev kernel with $\varepsilon=10$.}}
\begin{tabular}{ccccccccc}
\hline \hline                 &&&& \small{ $\tau=3$} &&&&\\\hline\label{table-1}
$h$ & \small{ ${\| e^h \|_{0,\Omega }}$}  &\small{order}& \small{${\| e^h\|_{0,\partial \Omega }}$}  &\small{order} & {${\| Le^h  \|_{0,\Omega }}$} & \small{order}& \small{$|||e^h|||$}& \small{order}\\\hline
\small{0.25}       &\small{4.0771e-01}   &  \small{ -}       &\small{9.4412e-01}    &\small{-}   &\small{9.3784e+00}
& \small{-}&\small{6.9802e+01}&\small{-} \\
\small{$h/2$}   &\small{1.3016e+00}   &\small{-1.6747}  &\small{2.5902e-01} &\small{1.8659}    &\small{1.7612e+01}  &\small{-0.9092}&\small{1.5023e+02}&\small{-1.1058} \\
\small{$h/4$}  &\small{1.4516e+00}   &\small{0.1574}   &\small{3.3942e-02}  &\small{2.9319}   &\small{1.6565e+01}  &\small{0.0884}&\small{1.5559e+02}&\small{-0.0506}\\
\small{$h/6$}  &\small{1.1913e+00}   &\small{0.4874}   &\small{1.3580e-02}  &\small{2.2093}   &\small{1.4866e+01}  &\small{0.2670}&\small{2.0259e+02}&\small{-0.6510}\\
\small{$h/8$}  &\small{7.9384e-01}   &\small{1.4111}   &\small{7.8229e-03}  &\small{1.9172}    &\small{1.1875e+01} &\small{0.7806}&\small{2.6821e+02}&\small{-0.9753}\\
\small{$h/10$}&\small{4.9996e-01}   &\small{2.0721}   &\small{4.7968e-03}  &\small{2.1920}    &\small{9.1629e+00} &\small{1.1621}&\small{3.1616e+02}&\small{-0.7369}\\
\small{$h/12$} &\small{4.3891e-01}  &\small{0.7142}   &\small{3.0758e-03}  &\small{2.4374}    &\small{8.6532e+00} &\small{0.3139}&\small{3.4881e+02}&\small{-0.5391}\\
\small{$h/14$} &\small{2.8627e-01} &\small{2.7725}    &\small{2.0804e-03}    &\small{2.5366}    &\small{6.8777e+00} &\small{1.4898}&\small{3.7222e+02}&\small{-0.4216}\\\hline
\small{Theory}      &\small{}   &   \small{-}       &\small{}    &\small{-}    &\small{}
& \small{-} &\small{}&\small{-}\\
\hline\hline
                        &&&&\small{  $\tau=4$}  &&&& \\\hline
\small{0.25}      &\small{3.5863e-01}   &   \small{-}       &\small{8.4776e-01}    &\small{-}    &\small{1.0102e+01}
& - &\small{6.3975e+01}&\small{-}\\
\small{$h/2$}  &\small{1.2905e+00}  &\small{-1.8473}   &\small{1.4867e-01}   & \small{2.5013}  &\small{1.5338e+01}  &\small{-0.6024}&\small{9.1456e+01}&\small{-0.5156}\\
\small{$h/4$}  &\small{7.7511e-01}  &\small{0.7354}    &\small{2.5646e-02}    &\small{2.5353}   &\small{1.1615e+01}  &\small{0.4011}&\small{1.1666e+02}&\small{-0.3512}\\
\small{$h/6$}  &\small{3.6300e-01}  &\small{1.8709}    &\small{8.9105e-03}    &\small{2.6073}   &\small{7.8091e+00}  &\small{0.9791}&\small{1.3099e+02}&\small{-0.2857}\\
\small{$h/8$}  &\small{1.1878e-01}  &\small{3.8834}    &\small{3.4407e-03}    &\small{3.3077}    &\small{4.4627e+00} &\small{1.9528}&\small{1.1720e+02}&\small{0.3867}\\
\small{$h/10$}&\small{4.3416e-02}  &\small{4.5102}    &\small{1.3764e-03}    &\small{4.1058}   &\small{2.6917e+00} &\small{2.2558}&\small{9.0872e+01}&\small{1.1445}\\
\small{$h/12$} &\small{2.5797e-02} &\small{2.8553}    &\small{7.4864e-04}    &\small{3.3402}    &\small{2.0585e+00} &\small{1.4708}&\small{8.4852e+01}&\small{0.3705}\\
\small{$h/14$} &\small{1.1324e-02} &\small{5.3411}    &\small{4.0271e-04}    &\small{4.0222}    &\small{1.3795e+00} &\small{2.5968}&\small{7.2102e+01}&\small{1.0562}\\\hline
\small{Theory}      &\small{}   &   \small{-}       &\small{}    &\small{-}    &\small{}
& \small{-} &\small{}&\small{-}\\
\hline\hline
                      &&&&\small{  $\tau=5$}  &&&& \\\hline
\small{0.25}                 &\small{4.5055e-01}   &    \small{-}        &\small{6.9867e-01}      &\small{-}    &\small{1.0848e+01}
&  - &\small{5.5563e+01}&\small{-}\\
\small{$h/2$} &\small{1.0304e+00}    &\small{-1.1935}  &\small{1.0777e-01}  &\small{2.6967}   &\small{1.3647e+01} &\small{-0.3312}&\small{6.8825e+01}&\small{-0.3088}\\
\small{$h/4$} &\small{2.9679e-01}     &\small{1.7957}   &\small{1.7418e-02}  &\small{2.6293}   &\small{7.1083e+00} &\small{0.9410}&\small{7.8402e+01}&\small{-0.1889}\\
\small{$h/6$} &\small{6.0793e-02}     &\small{3.9104}   &\small{3.8090e-03}  &\small{3.7492}   &\small{3.2229e+00} &\small{1.9508}&\small{5.5880e+01}&\small{0.8368}\\
\small{$h/8$} &\small{1.0943e-02}     &\small{5.9607}   &\small{9.6698e-04}  &\small{4.7656}   &\small{1.4055e+00} &\small{2.8846}&\small{3.3091e+01}&\small{1.8212}\\
\small{$h/10$} &\small{2.6877e-03}    &\small{6.2919}  &\small{2.8314e-04}  &\small{5.5042}  &\small{4.0552e-01}                &\small{3.0887}&\small{1.8827e+01}&\small{2.5275}\\\hline
\small{Theory}      &\small{}   &   \small{5}       &\small{}    &\small{4.5}    &\small{}
& \small{3} &\small{}&\small{3}\\
\hline\hline
                      &&&&\small{  $\tau=6$}  &&&& \\\hline
\small{0.25}              &\small{5.4390e-01}   &     -              &\small{5.9587e-01}      &  -    &\small{1.1171e+01}
&  - &\small{4.9306e+01}&\small{-}\\
\small{$h/2$ } &\small{7.7146e-01}   &\small{-0.5043}  &\small{8.1924e-02}  &\small{2.8626}  &\small{1.1857e+01} &\small{-0.0860}&\small{5.3803e+01}&\small{-0.1259}\\
\small{$h/4$ } &\small{8.6708e-02}   &\small{3.1534}  &\small{9.9357e-03}  &\small{3.0436}   &\small{3.9212e+00} &\small{1.5964}&\small{4.4618e+01}&\small{0.2701}\\
\small{$h/6$ } &\small{8.2390e-03}   &\small{5.8049}  &\small{1.3425e-03}  &\small{4.9366}   &\small{1.2420e+00} &\small{2.8355}&\small{1.9801e+01}&\small{2.0037}\\
\small{$h/8$ } &\small{9.3344e-04}   &\small{7.5700}  &\small{2.3671e-04}  &\small{6.0326}   &\small{4.3282e-01} &\small{3.6644}&\small{8.1892e+00}&\small{3.0690}\\\hline
\small{Theory}      &\small{}   &   \small{5}       &\small{}    &\small{4.5}    &\small{}
& \small{3} &\small{}&\small{3}\\
\hline
\end{tabular}
\end{center}
\end{table}

\begin{table}[]
\begin{center}
\caption{\small{Condition numbers and their orders of DLSP matrix for various values of $\tau$}} 
\begin{tabular}{ccccccccc}
\hline
\hline\label{table-2}
$h$ & \small{$\tau=3$ }  &\small{order}& \small{$\tau=4$ }  &\small{order} & {$\tau=5$ } & \small{order}& \small{$\tau=6$ }& \small{order}\\\hline
\small{0.25}        &\small{2.1442e+00}   &   \small{-}         &\small{2.6245e+00}  &\small{-}   &\small{2.7575e+00}
&\small{-}&\small{2.2632e+00}&\small{-} \\
\small{$h/2$}  &\small{6.1697e+00}   &\small{-1.5248}  &\small{5.0880e+01} &\small{-4.2770}    &\small{5.7017e+02}  &\small{--7.6919}&\small{5.9607e+03}&\small{-11.3629} \\
\small{$h/4$}  &\small{2.3461e+03}   &\small{-8.5708}   &\small{3.4348e+05}  &\small{-12.7208}   &\small{4.9499e+07}  &\small{-16.4056}&\small{6.8306e+09}&\small{-20.1281}\\
\small{$h/6$}  &\small{1.5702e+05}   &\small{-10.3674}   &\small{1.1296e+08}  &\small{-14.2939}   &\small{8.1722e+10}  &\small{-18.2732}&\small{5.8350e+10}&\small{-22.3261}\\
\small{$h/8$}  &\small{3.6020e+06}   &\small{-10.8901}   &\small{8.0450e+09}  &\small{-14.8280}    &\small{1.8394e+13} &\small{-18.8227}&\small{4.2756e+16}&\small{-22.9309}\\
\small{$h/10$}&\small{4.1449e+07}   &\small{-10.9480}   &\small{2.2430e+11}  &\small{-14.9138}    &\small{1.2496e+15} &\small{-18.9050}&\small{-}&\small{-}\\
\small{$h/12$} &\small{4.0986e+08}  &\small{-12.5677}   &\small{3.3015e+12}  &\small{-14.7496}    &\small{-}
&\small{-}&\small{-}&\small{-}\\
\small{$h/14$} &\small{2.1131e+09} &\small{-10.6396}    &\small{3.3761e+13}    &\small{-15.0822}    &\small{-} &\small{-}&\small{-}&\small{-}\\\hline
\end{tabular}
\end{center}
\end{table}
\newpage
\section{Conclusion}
A least-squares variational kernel-based method for solving the general second order elliptic problem with
nonhomogenous Dirichlet boundary conditions is given in this paper.
One of the attractive features of the method is that the approximating space
is not subject to the LBB condition. Besides, the discretization yields a positive definite system while the original PDE may not be symmetric at all. The approximation space is formed via kernels that reproduce Sobolev spaces as their native spaces.
We show that the DLSP formulations using sufficiently smooth kernels, which reproduce $H^\tau(\Omega)$,
can converge at the optimal rate in $H^t(\Omega)$-norm, with $4 - k \le t \le k$, where $\tau \ge k \ge 4$.
The condition number of the final least-squares system is also estimated in terms of the smoothness of the basis function and the discretization parameter.
Some parts of our analysis are subjected to a {\em conjecture} that demands an independent study in the theory of kernel approximations. See \eqref{eq1_231}.
Finally, we have reported some numerical results to confirm the theoretical bounds.
{As a downside, the condition numbers grow
at hight algebraic rates for smooth trial kernels. This paper does not concern special approaches or any preconditioning technique
to overcome this problem. However, we suggest some possible approaches here.
The stability estimates
may be greatly improved if a similar theory could be derived for {\em polyharmonic kernels} by scaling the points with mesh norm $h$ and carrying
the computation in the blown-up situation. More details can be found in \cite{Iske:2003-1,davydov-schaback:2019-1}.
As another possibility,
one can use the {\em localized bases for kernel space} \cite{Fuselier-et-all:2013}
instead of the global basis $U_{\Phi,X}$ to improve the condition numbers. The use of ``greedy algorithms" in trial space will be another possible approach \cite{wendland-schaback:2000,schaback:2014-1}. The compactly supported kernels in a multiscale setting can also be used to improve the conditioning at the price of a more computational cost \cite{Farrell-Wendland:2013}. 
Finally, the application of the method on the corresponding first order system of equations needs less smooth basis functions leading to a great improvement in the numerical conditioning.  
Since all the proposals above are rather involved and contain their own technical details, we do not peruse them further and leave them for
future studies.}

\end{document}